\documentclass[12pt]{article}
\usepackage{float}
\usepackage{makeidx}
\usepackage{latexsym}

\usepackage[margin=2.5cm]{geometry}
\usepackage[dvips]{graphicx}

\usepackage[english]{babel}
\usepackage{amssymb}
\usepackage{amsmath}
\usepackage{amsthm}
\usepackage{lineno}
\usepackage{xcolor}
\definecolor{blau}{rgb}{0.1,0.0,0.9}
\definecolor{gruen}{cmyk}{1.0,0.2,0.7,0.07}
\definecolor{mag}{cmyk}{0.0,0.9,0.3,0.0}

\newtheorem{theorem}{Theorem}[section]
\newtheorem{lemma}[theorem]{Lemma}
\newtheorem{corollary}[theorem]{Corollary}

\theoremstyle{definition}

\begin{document}
\date{\today}
\title{Restricted extension of sparse partial edge colorings of complete graphs}

\author{
{\sl Carl Johan Casselgren}\footnote{Department of Mathematics, Link\"oping University, SE-581 83 Link\"oping, Sweden. \newline
{\it E-mail address:} carl.johan.casselgren@liu.se  \, Casselgren was supported by a grant from the Swedish
Research Council (2017-05077)}
\and {\sl Lan Anh Pham }\footnote{Department of Mathematics, 
Ume\aa\enskip University, 
SE-901 87 Ume\aa, Sweden.
{\it E-mail address:} lan.pham@umu.se
}
}
\maketitle

\bigskip
\noindent
{\bf Abstract.}
Given a partial edge coloring of a complete graph $K_n$ and lists of allowed
colors for the non-colored edges of $K_n$, can we extend the partial edge coloring
to a proper edge coloring of $K_n$ using only colors from the lists?
We prove that this question has a positive answer in the case when both the
partial edge coloring and the color lists satisfy certain sparsity conditions.

\bigskip


\section{Introduction}

	An {\em edge precoloring} (or {\em partial edge coloring}) 
	of a graph $G$ is a proper edge coloring of some 
	subset $E' \subseteq E(G)$; {\em a $t$-edge precoloring}
	is such a coloring with $t$ colors.	
	A $t$-edge precoloring $\varphi$ is
	{\em extendable} if there is a proper $t$-edge coloring $f$
	such that $f(e) = \varphi(e)$ for any edge $e$ that is colored
	under $\varphi$; $f$ is called an {\em extension}
	of $\varphi$. 
	
	In general, the problem of extending a given edge precoloring
	is an $\mathcal{NP}$-complete problem, 
	already for $3$-regular bipartite graphs \cite{Fiala}.
	
	Questions on extending a partial edge coloring seems to have
	been first considered for balanced complete bipartite graphs, and
	these questions are usually referred to as the problem of
	completing partial Latin squares.
	In this form the problem appeared already in 1960, when Evans 
	\cite{Evans}  stated his now classic  conjecture  that for
	every positive integer $n$, if %only 
	$n-1$ edges in $K_{n,n}$ have been (properly) colored, 
	then the partial coloring can be extended to 
	a proper $n$-edge-coloring of $K_{n,n}$. 
	This conjecture was
	solved for large %enough 
	$n$ by H\"aggkvist  \cite{Haggkvist78} and later for all $n$ by Smetaniuk 
	\cite{Smetaniuk}, 
	and independently by Andersen and Hilton \cite{AndersenHilton}.
	Similar questions have also been investigated for complete graphs
	\cite{AndersenHilton2}.
	Moreover, quite recently, Casselgren et al. \cite{CasselgrenMarkstromPham}
	proved an analogue of this result
	for hypercubes.
	
	Generalizing this problem, Daykin and H\"aggkvist \cite{DH}
proved several results on
extending partial edge colorings of $K_{n,n}$, and they also conjectured that  much denser partial colorings  can be extended, as long as the 
colored edges are spread out in a specific sense:
%In particular, say that 
a partial $n$-edge coloring of $K_{n,n}$ is 
{\em $\epsilon$-dense} if there are at most $\epsilon n$ colored edges 
from $\{1,\dots,n\}$
at any vertex and each color in $\{1,\dots,n\}$
is used at most $\epsilon n$ times
in the partial coloring. %whole graph. 
Daykin and H\"aggkvist \cite{DH} conjectured
that for every positive integer $n$,
every $\frac{1}{4}$-dense partial proper $n$-edge
coloring can be extended
to a proper $n$-edge coloring of $K_{n,n}$,
and proved a version of the conjecture for $\epsilon=o(1)$ 
(as $n \to \infty$) and $n$ divisible by 16.
Bartlett \cite{Bartlett} proved  that this conjecture holds for 
a fixed positive $\epsilon$, and recently a 
different proof which improves the value of $\epsilon$ was 
given in \cite{BKLOT}.

For general edge colorings of  balanced complete
bipartite graphs, Dinitz conjectured,
and Galvin proved \cite{Galvin}, that if each edge of $K_{n,n}$
is given a list of $n$ colors, then there is a proper edge 
coloring of $K_{n,n}$ with support in the lists. 
Indeed, Galvin's result was a complete solution
of the well-known List Coloring Conjecture
for the case of bipartite multigraphs
(see e.g. \cite{haggkvist1992some} for more background on this conjecture and its relation
to the Dinitz' conjecture).

Motivated by the Dinitz' problem,  H\"aggkvist \cite{Ha89} 
introduced the notion of {\em $\beta n$-arrays},
which correspond to list assignments $L$ of forbidden colors for 
$E(K_{n,n})$,
such that each edge $e$ of $K_{n,n}$ is assigned a list $L(e)$
of at most $\beta n$ forbidden colors from $\{1,\dots,n\}$, 
and at every vertex $v$ each color is forbidden on 
at most $\beta n$ edges incident to $v$;
we call 
such a list assignment for $K_{n,n}$ {\em $\beta$-sparse}. 
%meeting at any one vertex.  
If $L$ is a list assignment for $E(K_{n,n})$, then
a proper $n$-edge coloring $\varphi$ of $K_{n,n}$ {\em avoids}
the list assignment $L$ if $\varphi(e) \notin L(e)$ for 
every edge $e$ of $K_{n,n}$;
if such a coloring exists, then $L$ is {\em avoidable}.
H\"aggkvist conjectured that there exists a fixed 
$\beta>0$, in fact also that $\beta=\frac{1}{3}$,  such that
for every positive integer $n$,
every $\beta$-sparse list assignment for $K_{n,n}$ is avoidable.
%any $\beta n$-array is avoidable.   
That such a $\beta>0$ exists was proved
for %large 
even $n$
by Andr\'en  in her PhD thesis \cite{Andren2010latin},
and later for all %large enough
$n$ in \cite{AndrenCasselgrenOhman}.

Combining the notions of extending a sparse precoloring and 
avoiding a sparse list assignment,
Andr\'en et al. \cite{AndrenCasselgrenOhman} proved that %for suitable 
there are constants $\alpha> 0$ and $\beta> 0$, %one can in fact find 
such that for every positive integer $n$,
every $\alpha$-dense partial edge
coloring of $K_{n,n}$ can be extended %in such a way that
to a proper $n$-edge-coloring avoiding any given $\beta$-sparse
list assignment $L$, provided that no edge $e$ is precolored by a color
that appears in $L(e)$. Quite recently, Casselgren et al.
obtained analogous results for hypercubes \cite{CasselgrenMarkstromPham2}.

In this paper, we consider the corresponding problem for complete graphs.
As mentioned above, edge precoloring extension problems have previously been
considered for complete graphs. The type of questions that we are
interested in here, however, seems to be a hitherto quite unexplored line
of research.
To state our main result, we need to introduce some terminology.
If $n$ is even, let $m=2n-1$ and if $n$ is odd, let $m=2n$.

A partial edge coloring of $K_{2n}$ or $K_{2n-1}$ with colors $1,\dots,m$ is
{\em $\alpha$-dense} if
\begin{itemize}
\item[(i)] every color appears on at most $\alpha m$ edges;
\item[(ii)] for every vertex $v$, at most $\alpha m$ edges incident to $v$ are precolored.
\end{itemize}

A list assignment $L$ for a complete graph $K_{2n}$ or $K_{2n-1}$ 
on the color set 
$\{1,\dots,m\}$ is  {\em $\beta$-sparse} if 
\begin{itemize}
\item[(i)] $|L(e)| \leq \beta m$ for every edge; %of $K_{2n}$ or $K_{2n-1}$;
\item[(ii)] for every vertex $v$, every color appears in the lists of at most $\beta m$ edges incident to $v$.
\end{itemize}

Our main result is the following.

\begin{theorem}
\label{mainth}
	There are constants $\alpha> 0$  and $\beta> 0$ such that  
	for every positive integer $n$, if $\varphi$  
	is an $\alpha$-dense $m$-edge precoloring of $K_{2n}$, $L$ is a
	$\beta$-sparse list assignment for $K_{2n}$ from the
	color set $\{1,\dots,m\}$,
	and $\varphi(e) \notin L(e)$ for every edge $e \in E(K_{2n})$,
	then there is a proper
	$m$-edge coloring of $K_{2n}$ which agrees with $\varphi$
	on any precolored edge and which avoids $L$.
\end{theorem}

Since any complete graph $K_{2n-1}$ of odd order is a subgraph of $K_{2n}$,
Theorem \ref{mainth} holds for any $\alpha$-dense $m$-edge precoloring and
any $\beta$-sparse list assignment for $K_{2n-1}$.
Hence, we have the following. For an integer $p$, we
define
$t= 4r-1$ if $p=4r$ or $p=4r-1$, and $t= 4r-2$ if $p=4r-2$ or $p =4r-3$.

\begin{corollary}
\label{cor:main}
	There are constants $\alpha> 0$  and $\beta> 0$ such that  
	for every positive integer $p$, if $\varphi$ is an $\alpha$-dense $t$-edge
	precoloring of $K_p$, $L$ is a $\beta$-sparse list assignment
	from the color set $\{1,\dots,t\}$, and $\varphi(e) \notin L(e)$ for every
	edge $e \in E(K_{p})$, then there is a proper $t$-edge coloring of $K_p$
	which agrees with $\varphi$ on any precolored edge and which avoids $L$.
\end{corollary}

Note that the number of colors in the preceding corollary is best possible
if $p \in \{4r, 4r-1\}$, while it remains an open question whether
$t=4r-2$ can be replaced by $t=4r-3$ if $p \in \{4r-2, 4r-3\}$.

The rest of the paper is devoted to the proof of Theorem
\ref{mainth}. The proof of this theorem is simlar to the proof of the
main result of \cite{AndrenCasselgrenMarkstrom}. However, we shall need to generalize
several tools from \cite{AndrenCasselgrenMarkstrom, Bartlett}
to the setting of complete graphs.

\section{Terminology, notation and proof outline}
	Let $\{p_1,p_2,\dots,p_n, q_1,q_2,\dots,q_n\}$ be 
	$2n$ vertices of the complete graph $K_{2n}$
	and $G_1$ ($G_2$) be the induced subgraph
	of $K_{2n}$ on the vertex set 
	$\{p_1,p_2,\dots,p_n\}$ ($\{q_1,q_2,\dots,q_n\}$) and
	$K_{n,n}$ be the complete bipartite graph
	%in which its vertices can be decomposed into
	with the partite sets $\{p_1,p_2,\dots,p_n\}$ and $\{q_1,q_2,\dots,q_n\}$.
	It is obvious that the complete graph $K_{2n}$
	is the union of the complete bipartite
	graph $K_{n,n}$ and the two copies $G_1$ and $G_2$ of the complete graph $K_n$.
	For any proper edge coloring $h$ of $K_{2n}$, we denote by $h_K$
	the restriction of this coloring to $K_{n,n}$; similarly,
	%the proper coloring of $K_{n,n}$ such that for every edge
	%$e \in K_{n,n} \cap K_{2n}$
	%we have $h(e)=h_K(e)$; 
	$h_{G_1}$ and $h_{G_2}$ are the restrictions of $h$
	to the subgraphs $G_1$ and $G_2$, respectively.
	
	For a vertex $u \in V(K_{2n})$, we denote by $E_u$ the set of 
	edges with one endpoint being $u$,
	and for a (partial) edge coloring $f$ of $K_{2n}$,
	let $f(u)$ denote the set of colors on the edges in $E_u$ under $f$.
	%If two edges $xy$ and $zt$ of $K_{2n}$ are in a cycle of length 
	%$4$ in $K_{2n}$, and all the vertices $x,y,z,t$ are distinct,
	%then the edges $xy$ and $zt$ are {\em parallel}.
	%For a vertex $u \in K_{n,n}$, let $K_u$ ($|K_u|=n$) 
	%be the maximum independent set that contains $u$. 
	Let $\varphi$ be an $\alpha$-dense precoloring of $K_{2n}$.
	%By definition, we have $|\varphi(u)| \leq \alpha m$.
	Edges of $K_{2n}$ which are colored under $\varphi$,
	are called {\em prescribed (with respect to $\varphi$)}.
	For an edge coloring $h$ of $K_{2n}$, %(or an edge coloring obtained from $h$),
	an edge $e$ of $K_{2n}$ is called
	{\em requested (under $h$ with respect to $\varphi$)}
	if $h(e) = c$ and
	$e$ is adjacent to an edge $e'$ such that $\varphi(e')=c$.
	
	Consider a $\beta$-sparse list assignment $L$ for $K_{2n}$.
	For an edge coloring $h$ of $K_{2n}$, %(or an edge coloring obtained from $h$),
	an edge $e$ of $K_{2n}$ is called a \textit{conflict edge (of $h$ with 
	respect to $L$)} if $h(e) \in L(e)$;
	such edges are also referred to as just {\em conflicts}.
	An \textit{allowed cycle (under $h$ 
	with respect to $L$)} of $K_{2n}$ is a $4$-cycle
	$\mathcal{C}=uvztu$ in $K_{2n}$ that is $2$-colored under $h$, and such that
	interchanging colors on $\mathcal{C}$ yields a proper edge coloring
	$h_1$ of $K_{2n}$ where none of $uv$, $vz$, 
	$zt$, $tu$ is a conflict edge. 
	We call such an interchange {\em a swap on $h$}, or a {\em swap on $\mathcal{C}$}.
\vspace{0.5cm}

\begin{enumerate}

	\item[Step I.] Define a {\em standard $m$-edge coloring $h$} 
	of the complete graph $K_{2n}$. In particular, this coloring has the property
	that ``most'' edges of $K_{n,n}$ are contained in a large number of
	$2$-colored $4$-cycles.

	\item[Step II.]
	Given the standard $m$-edge coloring $h$
	of $K_{2n}$, 
	from $h$ we construct a new
	%
	%we find a permutation $\rho$ of the elements of the set 
	%$\{1,\dots,m\}$
	%such that the 
	proper $m$-edge-coloring $h'$ %obtained
	%by applying $\rho$ to the colors used in $h$, 
	that satisfies
	certain sparsity conditions. These conditions shall be more precisely
	articulated below.

	\item[Step III.] From the precoloring $\varphi$ of $K_{2n}$, we
	define a new edge precoloring $\varphi'$ such that an edge 
	$e$ of $K_{2n}$ is colored 
	under $\varphi'$ if and only if $e$ is colored under 
	$\varphi$ or $e$ is a conflict edge
	of $h'$ with respect to $L$. We shall also require that
	each of the colors in $\{1,\dots,m\}$ is used a bounded number of times 
	under $\varphi'$.
	
	\item[Step IV.] 
	In this step we prove a series of lemmas which roughly implies that
	for almost all pair of edges $e$ and $e'$ in $K_{2n}$, we can construct
	a new edge coloring $h^T$ from $h'$ (or a coloring obtained from $h'$)
	such that $h^T(e')=h'(e)$ by recoloring a ``small'' subgraph of $K_{2n}$.

	\item[Step V.] Using the lemmas proved in the previous step,
	we shall in this step from $h'$ construct a coloring $h_q$ of $K_{2n}$
	that agrees with $\varphi'$ and which avoids $L$.
	This is done iteratively by steps:
	in each step we consider a prescribed edge $e$
	of $K_{2n}$, such that $h'(e) \neq \varphi'(e)$, and
	construct a subgraph $T_e$
	of $K_{2n}$, such that performing a series of swaps on allowed
	cycles, all edges of which are in $T_e$, we obtain a coloring
	$h''_1$ where $h''_1(e) = \varphi'(e)$.
	Hence, after completing this
	iterative procedure we obtain a coloring that is an extension of $\varphi'$
	(and thus $\varphi$), and which avoids $L$.

\end{enumerate}

	In Step IV and V we shall generalize several tools from
	\cite{AndrenCasselgrenMarkstrom, Bartlett} to the setting of complete graphs.

\section{Proof}

In this section we prove Theorem \ref{mainth}. In the proof we shall verify
that it is possible to perform Steps I-V described above
to obtain a proper $m$-edge-coloring of $K_{2n}$ that is an extension of $\varphi$
and which avoids $L$. This is done by proving some lemmas in each step.

The proof of Theorem \ref{mainth} involves a number of 
functions and parameters:
$$\alpha, \beta, d, \epsilon, k,  c(n), f(n)$$
and a number of inequalities that they must satisfy. For the reader's convenience,
explicit choices for which the proof holds are presented here:
$$\alpha = \frac{1}{1000000}, \quad \beta=\frac{1}{1000000}, \quad d= \frac{1}{200}, 
\quad
\epsilon =\frac{1}{50000},$$
$$k = \frac{1}{5000}, \quad c(n) = \left\lfloor\frac{n}{50000}\right\rfloor, \quad
f(n) = \left\lfloor\frac{n}{10000}\right\rfloor.$$
We shall also use the functions
$$c'(n) = c(n)/2, \quad  H(n) = 9 \alpha m + 9 f(n) + 6 c(n) + 4dn, \quad 
P(n)=dn+ \alpha  m + f(n).$$

Furthermore, we shall assume that $n$ is large enough whenever necessary. Since the
proof contains a finite number of inequalities that are valid if $n$ is large
enough, say $n \geq N$,
this suffices for proving the theorem with $\alpha'$ and $\beta'$
in place of $\alpha$ and $\beta$, and where we set 
$\alpha' = \min\{1/N, \alpha\}$
and $\beta' = \min\{1/N, \beta\}$.

We remark that since the numerical values of $\alpha$ and $\beta$ are not anywhere near
what we expect to be optimal, we have not put an effort into choosing optimal values
for these parameters. Since $K_{n,n}$ is a subgraph of $K_{2n}$, any upper bounds
on $\alpha$ and $\beta$ for the corresponding problem on 
complete bipartite graphs are also valid in the setting of complete graphs;
see \cite{AndrenCasselgrenMarkstrom} for a more
elaborate discussion on this question.

Finally, for simplicity of notation,
we shall omit floor and celling signs whenever these are not crucial.

\begin{proof}[Proof of Theorem \ref{mainth}]
Let $\varphi$ be an $\alpha$-dense precoloring of $K_{2n}$, and let $L$
be a $\beta$-sparse list assignment for $K_{2n}$
such that $\varphi(e) \notin L(e)$ for every edge $e \in E(K_{2n})$.
%Moreover, let $h$ be the standard $m$-edge coloring defined above. 

\bigskip

\noindent
{\bf Step I:} Below we shall define the {\em standard $m$-edge coloring $h$}
of the complete graph $K_{2n}$ by
defining an $n$-edge %standard
coloring for $K_{n,n}$ using the set of colors $\{1,2,\dots,n\}$ and
a $(m-n)$-edge 
%standard 
coloring for $G_1$ and $G_2$ using the set
of colors $\{n+1,\dots,m\}$. 
Throughout this paper, we assume $x \mod k =k$ in the case when $x \mod k \equiv 0$.

Firstly, we define a proper $n$-edge coloring for $K_{n,n}$ 
using the set of colors $\{1,2,\dots,n\}$.  
This coloring was used in \cite{AndrenCasselgrenOhman, AndrenCasselgrenMarkstrom,
Bartlett}, and
we shall give the explicit construction for the case when $n$ is even.
For the case $n$ is odd, one can modify the construction in the even case by swapping 
on some $2$-colored $4$-cycles and using a transversal; the details are given in Lemma 2.1 in  \cite{Bartlett}.

So suppose that $n=2r$.
For $1\leq i, j \leq n$, the standard coloring $h_K$ for $K_{n,n}$ is defined as follows.
\begin{equation}
h_K(p_iq_j) = \left\{ \begin{array}{llcl}
 j-i +1 & \mod r & \mbox{for}  & i,j \leq r,\\
 i-j +1 & \mod r & \mbox{for} & i, j> r, \\
(j-i+1 &\mod r) +r & \mbox{for} & i \leq r, j>r,\\
 (i-j+1 &\mod r) + r & \mbox{for} & i > r, j \leq r.
\end{array}\right.
\end{equation}
%
%A cycle $C$ in $K_{n,n}$ is {\em $2$-colored 4-cycle} if it contains four edges and its edges 
%are properly colored by two distinct colors $c_1, c_2$ from $\{1,\dots,n\}$.
The following property of $h_K$ is fundamental for our proof.
If a $2$-colored $4$-cycle with colors $c_1$ and $c_2$
satisfies that
$$|\{c_1, c_2\} \cap \{1,\dots, r\}|=1$$
then $C$ is called a {\em strong} $2$-colored 4-cycle.

\begin{lemma}
\cite{AndrenCasselgrenOhman,AndrenCasselgrenMarkstrom, Bartlett} 
Each edge in $K_{n,n}$ belongs to exactly $r$ distinct strong $2$-colored $4$-cycles
under $h_K$.
\end{lemma}
For the case when $n = 2r + 1$, we can construct an $n$-edge 
coloring $h_K$ for $K_{n,n}$ 
such that all but at most $3n+7$ edges are in $\left \lfloor{\frac{n}{2}}\right \rfloor$ strong $2$-colored 4-cycles.
In particular, there is a vertex in $K_{n,n}$ where no edge belongs 
to at least $\left \lfloor{\frac{n}{2}}\right \rfloor$  strong
$2$-colored $4$-cycles. The full proof appears in \cite{Bartlett} and therefore we omit the details here.

Secondly, let us define $(m-n)$-edge colorings of $G_1$ and $G_2$ using the set of colors $\{n+1,\dots,m\}$. 
Suppose first that $n$ is odd, and recall that $m=2n$. 
We define the colorings $h_{G_1}$ of $G_1$ and $h_{G_2}$ of $G_2$
by,
for $1\leq i, j \leq n$, 
setting
$$h_{G_1}(p_ip_j)=h_{G_2}(q_iq_j)=(i+j \mod n) + n.$$

Assume now that $n$ is even, and recall that $m=2n-1$.
We define the colorings
$h_{G_1}$ of $G_1$ and $h_{G_2}$ of $G_2$ as follows:

\begin{itemize}
\item $h_{G_1}(p_ip_j)=h_{G_2}(q_iq_j)=(i+j \mod n-1) + n$ for $1\leq i, j \leq n-1$.
\item $h_{G_1}(p_ip_n)=h_{G_2}(q_iq_n)=(2i \mod n-1) + n$ for $1\leq i, j \leq n-1$.
\end{itemize}

It is straightforward to verify that $h_K$, $h_{G_1}$, $h_{G_2}$ are proper colorings.
Taken together, the colorings $h_K$, $h_{G_1}$, $h_{G_2}$ constitute the
standard $m$-edge coloring $h$ of $K_{2n}$.
%Let $h$ be the $m$-edge standard coloring of $K_{2n}$ such that $h=h_K \cup h_{G_1} 
%\cup h_{G_2}$,
%then $h$ is a proper coloring.

\bigskip

\noindent
{\bf Step II:} Let $h$ be the $m$-edge coloring of $K_{2n}$ obtained in Step I, and let $\rho=(\rho_1,\rho_2)$ 
be a pair of permutations chosen 
independently and uniformly at random from all $n!$ permutations of 
the vertex labels of $G_1$ and $n!$ permutations of the vertex labels of $G_2$.
We permute the labels of the vertices with respect to the coloring of $h$,
while $\varphi$ is considered as a fixed partial coloring of $K_{2n}$.
Thus we can view a relabeling of the vertices in $G_1$ and $G_2$ 
with respect to $h$
(while
keeping colors of edges fixed)
as equivalent to defining a new proper edge coloring of $K_{2n}$ from $h$
by recoloring edges in $K_{2n}$. Hence, 
we can
think of $\rho$ as being applied to the edge coloring $h$ of $K_{2n}$
thereby defining a new edge coloring of $K_{2n}$
(rather than permuting vertex labels).

Denote by $h'$ a random $m$-edge coloring obtained from $h$ by applying $\rho$ to $h$.
Note that if $u'=\rho(u)$ and $v'=\rho(v)$, then $h'(u'v')=h(uv)$.

%We use the following lemma for constructing a required $m$-edge-coloring $h'$ from $h$.

%LEMMA 0.1
\begin{lemma}
	\label{alpha}
         Suppose that 
				$\alpha, \beta, \epsilon$ are constants, and $c(n)$ and $c'(n)=c(n)/2$ are
				functions of $n$, 
				such that
				%$c'(n)=c(n)/2$,
          $n-1>2c(n)>4$ and
          $$\Big( \dfrac{4\beta}{\epsilon - 4\beta}\Big)^{\epsilon - 4\beta} 
					\Big( \dfrac{1}{1 - 2\epsilon + 8\beta}\Big)^{1/2-\epsilon + 4\beta} <1,$$
          $$\alpha, \beta < \dfrac{c(n)}{2(n-c(n))} 
					\Big(\dfrac{n - c(n)}{n}\Big)^{\frac{n}{c(n)}}, \, \text{ and }$$
           $$\beta<\dfrac{c'(n)}{2(n-c'(n))} 
					\Big(\dfrac{n - c'(n)}{n}\Big)^{\frac{n}{c'(n)}}.$$
         % $$\beta < \dfrac{c(n)}{2(n-c(n))} \Big(\dfrac{n - c(n))}{n}\Big)^{\frac{n}{c(n)}},$$
       Then the probability that $h'$ fails the following 
			conditions tends to $0$ as $n \rightarrow \infty$.
        \begin{itemize}
	\item[(a)] All edges in $K_{n,n}$, except for $3n+7$, belong to at least 
	$\left \lfloor{\frac{n}{2}}\right \rfloor - \epsilon n$ allowed strong $2$-colored 4-cycles.
	
	\item[(b)] Each vertex of $K_{n,n}$ is incident to at most $c'(n)$ conflict edges in $K_{n,n}$.
	
	\item[(c)] For each color $c \in \{1,2,\dots,n\}$, 
	there are at most $c(n)$ edges in $K_{n,n}$ 
	that are colored $c$ that are conflicts.
	
	\item[(d)] For each color $c \in \{1,2,\dots,n\}$, 
	there are at most $c(n)$ edges in $K_{n,n}$ 
	that are colored $c$ that are prescribed.
	
	\item[(e)] 	For each pair of colors $c_1\in \{1,2,\dots,m\}$ and 
	$c_2 \in \{1,2,\dots,n\}$, there are at most $c(n)$ edges $e$ in $K_{n,n}$ 
	with color $c_2$ such that $c_1 \in L(e)$. %belongs to the corresponding list of 
	%forbidden colors in $L$.
	
	\item[(f)] Each vertex of $G_1$ $(G_2)$ is incident to at most $c'(n)$ conflict edges in $G_1$ $(G_2)$.

	\item[(g)] For each color $c \in \{n+1,n+2,\dots,m\}$, 
	there are at most $c(n)$ edges in $G_1$ $(G_2)$ 
	that are colored $c$ that are conflicts.
	
	\item[(h)] For each color $c \in \{n+1,n+2,\dots,m\}$, 
	there are at most $c(n)$ edges in $G_1$ $(G_2)$ 
	that are colored $c$ that are prescribed.
	
	\item[(i)] For each pair of colors $c_1\in \{1,2,\dots,m\}$ and 
	$c_2 \in \{n+1,n+2,\dots,m\}$, 
	there are at most $c(n)$ edges $e$ in $G_1$ $(G_2)$ with color $c_2$ 
	such that $c_1 \in L(e)$.
	%belongs to the corresponding list of forbidden colors in $L$.
	\end{itemize}
		\end{lemma}
		
	Note that the conditions in the lemma imply that the following holds for
	the coloring $h'$.
	\begin{itemize}
	\item[(a')]  Each vertex of $K_{2n}$ is incident to at most $c(n)$ conflict edges;
	\item[(b')]  For each color $c \in \{1,2,\dots,m\}$, 
	there are at most $c(n)$ edges in $K_{2n}$ 
	that are colored $c$ that are conflicts $($prescribed$)$;
	\item[(c')]  For each pair of colors $c_1, c_2 \in \{1,2,\dots,m\}$, 
	there are at most $c(n)$ edges $e$ in $K_{2n}$ 
	with color $c_2$ such that $c_1 \in L(e)$. 
	%belongs to the corresponding 
	%list of forbidden colors in $L$.
	\end{itemize}
	%By setting $\tau=1-C$, we can conclude that each edge of $K_{2n}$ 
	%belongs to at most $\tau m$ non-allowed cycles.

For the proof of this lemma we shall use the following theorem,
see e.g. \cite{AndrenCasselgrenMarkstrom}.

\begin{theorem}
\label{balancedbipartite}
If $B$ is a balanced bipartite graph on $2n$ vertices and $d_1,\dots,d_n$
are the degrees of the vertices in one part of $B$, then the number of perfect matchings in $B$
is at most $\prod_{1 \leq i \leq n} (d_i!)^{1/d_i}$.
\end{theorem}

\begin{proof}[Proof of Lemma \ref{alpha}]
Let $\alpha'=2\alpha$ and $\beta'=2\beta$; then the 
$\alpha$-dense precoloring $\varphi$ satisfies that
\begin{itemize}
\item[(I)] every color appears on at most $\alpha' n$ edges;
\item[(II)] for every vertex $v$, at most $\alpha' n$ edges incident with $v$ are precolored.
\end{itemize}

For the $\beta$-sparse list assignment $L$, we have
\begin{itemize}
\item[(III)] $|L(e)| \leq \beta' n$ for every edge of $K_{2n}$;
\item[(IV)] for every vertex $v$, every color appears in the lists 
of at most $\beta' n$ edges incident to $v$.
\end{itemize}

By applying Lemmas 3.2, 3.3, 3.4 in \cite{AndrenCasselgrenMarkstrom}, we deduce
that
the probability that $h'$ fails conditions (a), (b), (c), (d) or (e) tends to $0$ as $n \rightarrow \infty$ if
$$\Big( \dfrac{2\beta'}{\epsilon - 2\beta'}\Big)^{\epsilon - 2\beta'} \Big( \dfrac{1}{1 - 2\epsilon + 4\beta'}\Big)^{1/2-\epsilon + 2\beta'} <1;$$
$$\alpha', \beta'<\dfrac{c(n)}{(n-c(n))} \Big(\dfrac{n - c(n)}{n}\Big)^{\frac{n}{c(n)}};
\beta'<\dfrac{c'(n)}{(n-c'(n))} \Big(\dfrac{n - c'(n)}{n}\Big)^{\frac{n}{c'(n)}}.$$
Since all of these inequalities are true, our remaining job is to prove that
the probability that
$h'$ fails conditions (f), (g), (h) or (i) tends to $0$ as $n \rightarrow \infty$.

\vspace{0.3cm}
\begin{itemize}
\item We first prove (f) for $G_1$. Given a vertex $u \in G_1$,
it is obvious that $|E_u \cap E(G_1)|=n-1$.
We estimate the number of choices for the 
pair $\rho=(\rho_1,\rho_2)$ such that under $h'$ at least $c'(n)$ edges
in $E_u \cap E(G_1)$ are conflicts with $L$. 
There are $n!$ ways of choosing the permutation $\rho_2$;
fix such a permutation $\rho_2$. Also, there are $n$ choices for
a vertex $u_0$ such that $\rho_1(u_0)=u$,
we fix such a vertex $u_0$. 

Next, let $N$ be a subset of $E_{u_0} \cap E(G_1)$ such that $|N|=c'(n)$ and all edges in $N$ are mapped
to conflict edges by $\rho$, there are ${n-1} \choose c'(n)$ ways to chose $N$.
Next, let $B$ be a balanced bipartite graph defined as follows: 
the parts of $B$ are $E_{u_0} \cap E(G_1)$
and $E_u \cap E(G_1)$ and there is an edge between $u_0 x \in E_{u_0} \cap E(G_1)$ and $uy \in E_{u} \cap E(G_1)$ if
\begin{itemize}
\item $u_0x \notin N$, or
\item $u_0x \in N$ and $h(u_0x) \in L(uy)$.
\end{itemize}
If $u_0x \notin N$, then the degree of $u_0x$ in $B$ is $n-1$.
If $u_0x \in N$, then the degree of $u_0x$ in $B$ is at most $\beta'n$ because the color $h(u_0x)$
occurs at most $\beta'n$ times in $E_u \cap E(G_1)$.
A perfect matching in $B$ corresponds to a choice of $\rho_1$ so that $u_0$ is mapped to $u$ and
all edges in $N$ are mapped to conflict edges under $h'$. By Theorem 
\ref{balancedbipartite},
the number of perfect matchings in $B$ is at most
$$\big((\beta'n)!\big) ^{\frac{c'(n)}{\beta'n}} \big((n-1)!\big) ^{\frac{n-1-c'(n)}{n-1}}.$$

So the probability that $E_u \cap E(G_1)$ contains at
least $c'(n)$ conflicts with $L$ is at most
$$X=\dfrac{n! n {{n-1} \choose c'(n)} \big((\beta'n)!\big) ^{\frac{c'(n)}{\beta'n}} 
\big((n-1)!\big) ^{\frac{n-1-c'(n)}{n-1}}}{(n!)^2}$$
$$=\dfrac{\big((\beta'n)!\big) ^{\frac{c'(n)}{\beta'n}} \big((n-1)!\big) ^{\frac{n-1-c'(n)}{n-1}}}{c'(n)! (n-1-c'(n))!}$$
Using Stirling's formula
$$n! =C_0n^{a_0} \Big(\frac{n}{e}\Big)^n$$ 
for some positive constants $C_0$ and $a_0$; we have:
$$X \leq Cn^a \Big( \dfrac{\beta'n}{c'(n)}\Big)^{c'(n)} 
\Big(\dfrac{n-1}{n-1-c'(n)}\Big)^{(n-1-c'(n))}$$
$$=Cn^a \Big( \dfrac{\beta'n}{c'(n)}\Big)^{c'(n)} \Big(1+\dfrac{1}{(n-1-c'(n))/c'(n)}\Big)^{(n-1-c'(n))}$$
where $C$ and $a$ are some positive constants.

Since the function $f(x)=(1+\dfrac{1}{x})^{xy}$ is increasing for $x, y\geq 1$, we have
$$X<Cn^a \Big( \dfrac{\beta'n}{c'(n)}\Big)^{c'(n)} \Big(1+\dfrac{1}{(n-c'(n))/c'(n)}\Big)^{(n-c'(n))}$$
$$=Cn^a \Big( \dfrac{\beta'n}{c'(n)}\Big)^{c'(n)} \Big(\dfrac{n}{n-c'(n)}\Big)^{n-c'(n)}$$
%$$=Cn^a \Big(  \dfrac{\beta'n}{c'(n)}\Big)^{c'(n)}   \Big(\dfrac{n-c'(n)}{n}\Big)^{c'(n)}  \Big(\dfrac{n}{n-c'(n)}\Big)^{n}$$
$$=Cn^a \Big( \dfrac{\beta'(n-c'(n))}{c'(n)}\Big)^{c'(n)} \Big(\dfrac{n}{n-c'(n)}\Big)^{n}$$

Now, since $G_1$ has $n$ vertices, the probability that $h'$ 
fails condition (f) for $G_1$ is at most
$$Y=Cn^{a+1} \Big( \dfrac{\beta'(n-c'(n))}{c'(n)}\Big)^{c'(n)} \Big(\dfrac{n}{n-c'(n)}\Big)^{n}$$

Since  $\beta'<\dfrac{c'(n)}{(n-c'(n))} 
\Big(\dfrac{n - c'(n)}{n}\Big)^{\frac{n}{c'(n)}}$,
 we have
$Y \rightarrow 0$ as  $n \rightarrow \infty$; thus the probability that $h'$ fails condition (f) for $G_1$
tends to zero as $n \rightarrow \infty$. 
That the probability that $h'$ fails condition (f) for $G_2$
tends to zero as $n \rightarrow \infty$ can be proved similarly.

\item Next, we prove (g) for $G_1$. Let $n_0=(n-1)/2$ if $n$ is odd and $n_0=n/2$ if $n$ is even.
The total number of edges in $G_1$ is $n(n-1)/2$. Since the total number of colors used to color
the edges of $G_1$ under $h$ is $n$ if $n$ is odd, and $n-1$ if $n$ is even, 
there are exactly $n_0$ edges in $G_1$ that 
are colored by a fixed color in $\{n+1, n+2, \dots, m\}$.

Let $c$ be a color in $\{n+1, n+2, \dots, m\}$ and 
$P_c=\{x_1y_1, x_2y_2,\dots,x_{n_0}y_{n_0}\}$ be the set of edges
that are colored $c$ under $h$. We estimate the number of 
choices for the
pair $\rho=(\rho_1,\rho_2)$ such that under $h'$
at least $c(n)$ edges in $G_1$ that are colored $c$ 
are conflicts with $L$.
There are $n!$ ways of choosing the permutation $\rho_2$; 
fix such a permutation $\rho_2$. 

\begin{itemize}

\item If $n$ is odd, there is only one vertex $u$ in $G_1$ that is 
%different to the vertices in the set 
not contained in the set
$V'_c=\{x_1,x_2,\dots,x_{n_0}, y_1,y_2,\dots,y_{n_0}\}$; there are $n$ ways of
choosing a vertex $u_1$ 
such that $\rho_1(u)=u_1$, we fix such a vertex $u_1$. 
Moreover, there
are $(n-1)\dots(n-n_0)$ ways of choosing
a set of vertices $\{x'_1,\dots,x'_{n_0}\}$ satisfying 
that $\rho_1(x_i)=x'_i$, $i=1,\dots,n_0$; fix such a set
$\{x'_1,\dots,x'_{n_0}\}$ and let 
$Q=V(G_1)  \setminus \{x'_1,\dots,x'_{n_0}, u_1\}$. 
Note that $n \big((n-1) \dots (n-n_0)\big) =n!/n_0!$ (since $n_0=(n-1)/2$ if $n$ is odd).

\item If $n$ is even, then 
$V(G_1) = V'_c$, where the latter set is defined as above.
%the set of vertices of
%$G_1$ is $P'_c$, thus we do not have such a vertex $u$.
There are $n(n-1)\dots(n-n_0+1)$ ways of choosing
a set of vertices $\{x'_1,\dots,x'_{n_0}\}$ 
such that $\rho_1(x_i)=x'_i$, $i=1,\dots,n_0$; fix such a set
$\{x'_1,\dots,x'_{n_0}\}$ and let $Q=V(G_1)  \setminus \{x'_1,\dots,x'_{n_0}\}$. 
Note that $n(n-1)\dots(n-n_0+1) =n!/n_0!$.
\end{itemize}

Let $N$ be a subset of $P_c$ such that $|N|=c(n)$
and all edges in $N$ are mapped to conflict edges by $\rho$; there are 
$n_0 \choose c(n)$ ways to chose $N$.
We now define a balanced bipartite graph $B$ as follows: the parts of $B$ are $P_c$
and $Q$ and there is an edge between $x_iy_i \in P_c$ and $y'_j \in Q$ if
\begin{itemize}
\item $x_iy_i \notin N$, or
\item $x_iy_i \in N$ and $h(x_iy_i)=c \in L(x'_iy'_j)$.
\end{itemize}
If $x_iy_i \notin N$, then the degree of $x_iy_i$ in $B$ is $n_0$.
If $x_iy_i \in N$, then the degree of $x_iy_i$ in $B$ is at most $\beta' n$ because the color $c$
occurs at most $\beta' n$ times in $E_{x'_i} \cap E(G_1)$.
A perfect matching in $B$ corresponds to a choice 
of $\rho_1$ so that $\rho_1(x_i)=x'_i$, $i=1,\dots,n_0$ and
all edges in $N$ are mapped to conflict edges in $h'$. By Theorem 
\ref{balancedbipartite},
the number of perfect matchings in $B$ is at most
$$\big((\beta' n)!\big) ^{\frac{c(n)}{\beta' n}} \big(n_0 !\big) ^{\frac{n_0-c(n)}{n_0}}.$$

So the probability that $P_c$ contains at least $c'(n)$ conflicts with $L$ is at most
$$X=\dfrac{n! (n!/n_0!) {n_0 \choose c(n)} \big((\beta' n)!\big) ^{\frac{c(n)}{\beta' n}} \big(n_0!\big) ^{\frac{n_0-c(n)}{n_0}}}{(n!)^2}$$
$$=\dfrac{\big((\beta' n)!\big) ^{\frac{c(n)}{\beta'n}} 
\big(n_0!\big) ^{\frac{n_0-c(n)}{n_0}}}{c(n)! (n_0-c(n))!}$$
Using Stirling's formula and similar estimates as above we deduce
that
$$X \leq Cn^a \Big( \dfrac{\beta'n}{c(n)}\Big)^{c(n)} \Big(\dfrac{n_0}{n_0-c(n)}\Big)^{n_0-c(n)}$$
$$<Cn^a \Big( \dfrac{\beta'n}{c(n)}\Big)^{c(n)} \Big(\dfrac{n}{n-c(n)}\Big)^{n-c(n)}
=Cn^a \Big( \dfrac{\beta'(n-c(n))}{c(n)}\Big)^{c(n)} \Big(\dfrac{n}{n-c(n)}\Big)^{n}$$
where $C$ and $a$ are some positive constants.

Note that there are at most $n$ colors in $h_{G_1}$, thus the probability that $h'$ fails condition (g) for $G_1$ is at most
$$Y=Cn^{a+1} \Big( \dfrac{\beta'(n-c(n))}{c(n)}\Big)^{c(n)} \Big(\dfrac{n}{n-c(n)}\Big)^{n}.$$

Since  $\beta'<\dfrac{c(n)}{(n-c(n))} \Big(\dfrac{n - c(n))}{n}\Big)^{\frac{n}{c(n)}}$, we have
$Y \rightarrow 0$ as  $n \rightarrow \infty$; thus the probability that $h'$ fails condition (g) for $G_1$
tends to zero as $n \rightarrow \infty$. 
That the probability that $h'$ fails condition (g) for $G_2$
tends to zero as $n \rightarrow \infty$ can be proved similarly.

\item The proof of (h) is almost identical to the proof of (g) except that
we use the property that at most $\alpha' n$ edges incident to any vertex are prescribed. We omit the details.

\item The proof of (i) is also almost identical to the proof of (g);
here one also has to use the
property that any fixed color $c_1$ appears in the lists of at most of $\beta' n$ edges incident to any given vertex;
additionally instead of having $n$ choices for a 
color $c$ as in the proof of $(g)$, we will have $m(m-n)$ choices for a
pair of colors $(c_1,c_2)$. Here, as well, we omit the details.
\end{itemize}
\end{proof}

Lemma \ref{alpha} implies that there exists a pair of permutations 
$\rho=(\rho_1,\rho_2)$ such that 
if $h'$ is the proper $m$-edge coloring obtained from $h$ by applying $\rho$ to $h$ then $h'$
satisfies conditions (a)-(i) of Lemma \ref{alpha}.
\bigskip

\noindent
{\bf Step III:} Let $h'$ be the proper $m$-edge coloring satisfying conditions
(a)-(i) of Lemma \ref{alpha} obtained in the previous step.

We use the following lemma for extending $\varphi$ to a 
proper $m$-edge precoloring $\varphi'$ of $K_{2n}$, such that an
edge $e$ of $K_{2n}$ is colored under $\varphi'$ if and only if
$e$ is precolored under $\varphi$ or $e$ is a conflict edge of
$h'$ with $L$.

\begin{lemma}
	\label{gamma}
        Let $\alpha, \beta$ be constants and $c(n),f(n)$ be functions of $n$ such that
         $$m - \beta m -2 \alpha m -  2c(n) - \dfrac{2nc(n)}{f(n)} \geq 1.$$ 
        There is a proper $m$-edge precoloring $\varphi'$ of $K_{2n}$ satisfying the following:
        
        \begin{itemize}
        \item[(a)] $\varphi'(uv)=\varphi(uv)$ for any edge  $uv$ of $K_{2n}$ that is precolored under $\varphi$.
        
        \item[(b)] For every conflict edge $uv$ of $h'$ that is not colored 
				under $\varphi$, 
        $uv$ is colored under $\varphi'$ and $\varphi'(uv) \notin L(uv)$.
        
	\item[(c)] There are at most $\alpha m + c(n)$ prescribed edges at each vertex of $K_{2n}$ under $\varphi'$.
				
	\item[(d)] There are at most $\alpha m +f(n)$ prescribed edges with color $i$, $i =1,\dots, m$, under $\varphi'$.
	\end{itemize}
	Furthermore, the edge coloring $h'$ of $K_{2n}$  and the precoloring 
	$\varphi'$ of $K_{2n}$ satisfy that 
	\begin{itemize}
	\item[(e)] For each color $c \in \{1,2,\dots,n\}$, there are at most $2c(n)$ prescribed edges in $K_{n,n}$ with color $c$ under $h'$.
	
	\item[(f)] For each color $c \in \{n+1,n+2,\dots,m\}$, 
	there are at most $2c(n)$ prescribed edges in $G_1$ $(G_2)$ with color $c$ 
	under $h'$.
	\end{itemize}
	\end{lemma}
	
	Note that the two conditions (e) and (f) imply that
	\begin{itemize}
	\item[(g)] For each color $c \in \{1,2,\dots,m\}$, 
	there are at most $2c(n)$ prescribed edges in $K_{2n}$ with color $c$ in $h'$.
	\end{itemize}

	\begin{proof}
	We shall construct the coloring $\varphi'$ by assigning a color
	to every conflict edge; this is done by iteratively constructing
	an $m$-edge precoloring $\phi$ of the conflict edges of $K_{2n}$;
	in each step we color a hitherto uncolored conflict edge, thereby
	transforming a conflict edge to a prescribed edge.
		
	A color $c$ is {\em $\phi$-overloaded} in $K_{2n}$ if $c$ appears on at least 
	$f(n)$ edges in $K_{2n}$ under $\phi$. Since each vertex of $K_{2n}$ is incident 
	with at most $c(n)$ conflict edges, the number of conflict edges in $K_{2n}$ is 
	at most $2nc(n)$; this implies that at most 
	$\dfrac{2nc(n)}{f(n)}$ colors are $\phi$-overloaded in $K_{2n}$.
	
	Let $G$ be the subgraph of $K_{2n}$ induced by all conflict edges of $K_{2n}$.
	Let us now construct the $m$-edge coloring $\phi$ of $G$.
	We color the edges of $G$ by steps, and in each step we define a 
	list $\mathcal{L}(e)$ of allowed colors for 
	a hitherto uncolored edge $e =uv$ of $G$ by for 
	every color $c \in \{1,\dots, m\}$ 	including $c$ in $\mathcal{L}(e)$ if 
	\begin{itemize}
	\item $c \notin L(uv)$,
	\item $c$ does not appear in $\varphi(u)$ or $\varphi(v)$,
	or on any previously colored edge of $G$ that is adjacent to $e$,
	\item $c$ is not $\phi$-overloaded in $K_{2n}$.
	\end{itemize}
	
	Our goal is then to pick a color $\phi(e)$ from $\mathcal{L}(e)$ for $e$. 
	Given that this is possible for each edge of $G$, 
	this procedure clearly produces a $m$-edge-coloring $\phi$ of $G$, so that
	$\phi$ and $\varphi$ taken together form a proper $m$-edge precoloring of $K_{2n}$.
	Using the estimates above and the facts that $G$ has maximum degree $c(n)$, 
	and at most $\alpha m$ edges incident with any vertex $v$ of $K_{2m}$
	are prescribed with respect to $\varphi$,
	%$|\varphi(v)| \leq \alpha m$ for any vertex $v$ of $K_{2n}$, 
	we have
	$$\mathcal{L}(e) \geq m - \beta m -2 \alpha m -  2c(n) - \dfrac{2nc(n)}{f(n)}$$
	for every edge $e$ of $G$ in the process of constructing $\phi$, and 
	by assumption $\mathcal{L}(e) \geq 1$.
	Thus, we conclude that we can choose an allowed color for each conflict edge
	so that the coloring $\phi$ satisfies the above conditions.
	This implies that taking $\phi$ and $\varphi$ together we obtain
	a proper $m$-precoloring $\varphi'$ of the edges of $K_{2n}$.
	There are at most $\alpha m + c(n)$ prescribed edges at each vertex 
	of $K_{2n}$ under $\varphi'$
	because the maximum degree of $G$ is $c(n)$. The fact that we do not use 
	$\phi$-overloaded 
	colors in $\phi$ implies that there are at most $\alpha m +f(n)$ 
	prescribed edges with color $i$, 
	$i =1,\dots, m$, under $\varphi'$. 
	
	Let us next prove that the precoloring $\varphi'$ satisfies condition
	(e). By the previous lemma, for each color $c \in \{1,2,\dots,n\}$, 
	there are at most $c(n)$ edges in $K_{n,n}$ 
	that are colored $c$ that are prescribed under $\varphi$. Furthermore, 
	we have at most $c(n)$ edges in $K_{n,n}$ 
	that are colored $c$ that are conflict; thus after transforming all 
	conflict edges to prescribed edges, 
	there are at most $2c(n)$ prescribed edges in $K_{n,n}$ with respect to 
	$\varphi'$ that are colored $c$ under $h'$. 
	A similar argument shows that condition (f) holds as well. %Condition (g) is a combination of (e) and (f).
	\end{proof}

	\bigskip

	\noindent
	{\bf Step IV:} Let $h'$ be the $m$-edge coloring of
	$K_{2n}$ obtained in Step II, and
	suppose that $\hat h$
	is a proper $m$-edge coloring of $K_{2n}$ 
	obtained from $h'$ by performing a sequence of swaps. 
	We say that an edge $e$ in $K_{2n}$ is
	{\em disturbed (in $\hat{h}$)} if $e$ appears in a swap
	which is used for obtaining $\hat{h}$ from $h'$, or if $e$ is one of
	the original at most $3n+7$ edges in $h'$ that do not belong 
	to at least $\left \lfloor{\frac{n}{2}}\right \rfloor - \epsilon n$ 
	allowed strong $2$-colored 4-cycles in $h'$. 
	For a constant $d >0$, we say that a vertex $v$ or color $c$ is $d$-overloaded
	if at least $d n$ 
	edges which are incident to $v$ or colored $c$, respectively,  are disturbed.
	
	The following lemma is similar to Lemma $3.5$ and Lemma $3.6$ in 
	\cite{AndrenCasselgrenMarkstrom},
	which are strengthened variants of Lemma 2.2. in \cite{Bartlett}. Thus, we shall skip the proof.
	
	%LEMMASWAPEDGEOF_K
	\begin{lemma}
	\label{swapedgeinK}
	Suppose that $h''$ is a proper $m$-edge coloring of $K_{2n}$ obtained from $h'$
	by performing some sequence of swaps
	on $h'$ and that at most $kn^2$ edges in $h''$ are disturbed 
	for some constant $k>0$. Suppose that
	for each color $c$, at most $2c(n)+P(n)$ edges 
	with color $c$ under $h''$ are prescribed.
	Moreover, let $\{t_1,\dots,t_a\}$ be a set of colors from $h''$.
	If $$\left \lfloor{\frac{n}{2}}\right \rfloor - 2\epsilon n - 
	6d n - 5 \dfrac{k}{d}n - 4\alpha m - 8c(n) - 3a - 3\beta m - 2P(n)- 6 >0$$
	then for any vertex $u_1$ of $G_1$ $(G_2)$ and all but at most
	\begin{itemize}
	
	\item $2 \dfrac{k}{d} n + \alpha m +c(n)+a$ choices of a vertex $u_2$ in
	$G_2$ $(G_1)$, 
	such that $h''(u_1u_2) \in \{1,2,\dots,n\}$, and 
 	
	\item $4 \dfrac{k}{d} n+ a+1+4c(n)+2\beta m + 2 \alpha m +2dn+P(n)$ choices of 
	a vertex $v_2$ in $G_2$ $(G_1)$, such that $h''(u_1v_2) \in \{1,2,\dots,n\}$,
	
	\end{itemize}
	there is a subgraph $T$ of $K_{n,n}$ 
	and a proper $m$-edge coloring $h^T$ of $K_{2n}$, obtained from
	$h''$ by performing a sequence of swaps on $4$-cycles in
	$T$, that satisfies the following:
	\begin{itemize}
	\item the color of any edge of $T$ under $h''$ is not $d$-overloaded;
 	\item no edges that are prescribed (with respect to $\varphi'$) are in $T$;
	\item $h''$ and $h^T$ differs on at most $16$ edges 
	\emph{(}i.e. $T$ contains at most $16$ edges\emph{)};
	\item no edge with a color in $\{t_1,...,t_a\}$ under $h''$ is in $T$;
	\item $h^T(u_1u_2)=h''(u_1v_2)$ and  $h^T(u_1v_2)=h''(u_1u_2)$;
	\item if there is a conflict of $h^T$ with respect to $L$, then this edge is also
	a conflict of $h''$;
	\item any edge in $G_1$ or $G_2$ that is requested under $h^T$ (with respect
	to $\varphi'$) is also
	requested under $h''$.
	\end{itemize}
	\end{lemma}

	Lemma \ref{swapedgeinK} states that there are  
	many pairs of adjacent edges $e_x,e_y \in E(K_{n,n})$ satisfying that
	$h''(e_x), h''(e_y) \in \{1,2,\dots,n\}$ such that we can exchange their colors
	by recoloring a small subgraph of $K_{n,n}$.
	When applying the preceding lemma, we shall
	refer to $u_1u_2$ as the ``first edge'' and $u_1v_2$ as the ``second edge''.

	Given an edge $e_x \in E(K_{n,n})$ such that $h''(e_x) \in \{1,2,\dots,n\}$, 
	the following lemma
	is used for obtaining a coloring where an edge $e_y \in E(K_{n,n})$ 
	adjacent to $e_x$ is colored $h''(e_x)$.
	%and another edge 
	%$e_y \in K_{n,n}$ that shares an endpoint with $e_x$.

	%LEMMASWAPEDGEOF_K_1
	\begin{lemma}
	\label{swapedgeinK1}
	Suppose that $h''$ is a proper $m$-edge coloring of $K_{2n}$ 
	obtained from $h'$ by performing some sequence of swaps
	on $h'$ and that at most $kn^2$ edges in $h''$ are 
	disturbed for some constant $k>0$. Suppose that
	for each color $c$, at most $2c(n)+P(n)$ edges with color 
	$c$ under $h''$ are prescribed,
	and at most $H(n)$ edges with color $c$ are disturbed.
	Moreover, let $\{t_1,\dots,t_a\}$ be a set of colors from $h''$.
	If $$\left \lfloor{\frac{n}{2}}\right \rfloor - 2\epsilon n - 6d n - 
	5 \dfrac{k + 34/n^2}{d}n - 4\alpha m - 8c(n) - 3a - 3\beta m - 2P(n)- 6 >0$$
	and
	$$n -  \Big(8 \dfrac{k + 34/n^2}{d} n + 2a + 3 + 8c(n) + 6\beta m + 
	4 \alpha m + 4dn+ 2P(n) + H(n)\Big)>0$$
	then for any edge $u_1u_2$ of $K_{n,n}$ with
	$$h''(u_1u_2)=c_1, \; c_1 \in  \{1,2,\dots,n\}, \; 
	c_1 \notin  \{t_1,\dots,t_a\}$$ and all but at most
	$$4c(n) + P(n) + 2\beta m + 2\alpha m + 2a + 1 + 4 \dfrac{k+34/n^2}{d} n + H(n)$$ 
	choices of a vertex $v_2$ satisfying that $u_1v_2 \in E(K_{n,n})$, 
	there is a subgraph $T$ of $K_{n,n}$ and a proper $m$-edge coloring 
	$h^T$ of $K_{2n}$, obtained from
	$h''$ by performing a sequence of swaps on $4$-cycles in
	$T$, that satisfies the following:
	\begin{itemize}
	
	\item except $c_1$, any color of an edge in $T$ under $h''$ is not $d$-overloaded;
 	
	\item except $u_1u_2$, no edge in $T$ is prescribed;
	
	%at most one prescribed edge is in $T$;
	
	\item $h''$ and $h^T$ differs on at most $34$ edges \emph{(}i.e. $T$ 
	contains at most $34$ edges\emph{)};
	
	\item no edge with a color in $\{t_1,\dots,t_a\}$ under $h''$ is in $T$;
	
	\item $h^T(u_1v_2)=h''(u_1u_2)=c_1$;
	%\emph{(}in this lemma, $u_1v_2$ is called the second edge\emph{)};
	
	\item if there is a conflict of $h^T$ with $L$, then this edge is also
	a conflict of $h''$ with $L$;
	
	\item any edge in $G_1$ or $G_2$ that is requested under $h^T$ (with respect
	to $\varphi'$) is also
	requested under $h''$.
	\end{itemize}
	\end{lemma}
	
	%Proof_ LEMMASWAPEDGEOF_K_1
	\begin{proof}
	Without loss of generality, assume that $u_1 \in V(G_1)$; 
	this implies $u_2 \in V(G_2)$. 
		We choose $v_2 \in V(G_2)$ so that the following properties hold.
	
	\begin{itemize}
	\item The edge $v_1v_2$ in $K_{n,n}$ satisfying $h''(v_1v_2)=c_1$ is not 
	disturbed and not prescribed.
	Since there are at most $2c(n)+P(n)$ prescribed edges 
	and at most $H(n)$ disturbed edges with color $c_1$ under $h''$,
	and each such prescribed or disturbed edge of $K_{n,n}$ can be incident to at most 
	one vertex of $G_2$,
	this eliminates at most $2c(n)+P(n)+H(n)$ choices.
	\item The edge $u_1v_2$ and the edge $u_2v_1$ 
	are both valid choices for the first edge in an application of 
	Lemma \ref{swapedgeinK}. 
	This eliminates at most 
	$$2\Big(2 \dfrac{k+34/n^2}{d} n + \alpha m +c(n)+a\Big)$$ choices. 
	The additive factor $34/n^2$ comes from the fact that performing
	a sequence of swaps to transform $h''$ into $h^T$
	will create at most $34$ additional disturbed edges.
	\item $c_1 \notin L(u_1v_2) \cup L(u_2v_1)$ and $u_1u_2 \neq v_1v_2$. 
	This excludes at most $2\beta m + 1$ choices.
	\end{itemize}
	
	Thus we have at least
	$$n - 4c(n)- P(n) - 2\beta m - 2\alpha m - 2a -1 - 4 \dfrac{k+34/n^2}{d} n - H(n)$$	
	choices for a vertex $v_2$ and an edge $v_1v_2$. 
	We note that this expression is greater than zero by assumption, 
	so we can indeed make the choice.
	
	Next, we want to choose a color 
	$c_2 \in \{1,2,\dots,n\}$ such that the following properties hold.
	\begin{itemize}
	\item The edges $e_1$ and $e_2$ colored $c_2$ under $h''$
	that are incident with $u_1$ and $u_2$, respectively,
	%
	%that is adjacent to $u_1$ and to 
	%the edge $e_2$ that is adjacent to $u_2$ 
	%containing the same color $c_2$ are 
	are both valid choices for the second edge in 
	an application of Lemma \ref{swapedgeinK}; this eliminates at most 
	$$2\Big(4 \dfrac{k+34/n^2}{d} n+ a+1+4c(n)+2\beta m + 2 \alpha m +2dn+P(n)\Big)$$
	choices. %e_1 \neq e_2 since they contain color $c_2 \neq c_1$.
	Note that this condition implies that color $c_2$ is not $d$-overloaded.
	\item $c_2 \neq c_1$ and $c_2 \notin L(u_1u_2) \cup L(v_1v_2)$. 
	This excludes at most $2\beta m+1$ choices.

	\end{itemize}
	Thus we have at least
	$$n -  \Big(8 \dfrac{k + 34/n^2}{d} n + 2a + 3 + 8c(n) + 
	6\beta m + 4 \alpha m + 4dn+ 2P(n)\Big)$$
	choices. By assumption, this expression is greater than zero, 
	so we can indeed choose such color $c_2$. Now, since 
	$$\left \lfloor{\frac{n}{2}}\right \rfloor - 2\epsilon n - 6d n - 
	5 \dfrac{k+34/n^2}{d}n - 4\alpha m - 8c(n) - 3a - 3\beta m - 2P(n)- 6 >0,$$
	we can apply Lemma \ref{swapedgeinK} two consecutive times 
	to exchange the color of $u_1v_2$ and $e_1$, and
	similarly for $u_2v_1$ and $e_2$. Finally, 
	by swapping on the $2$-colored $4$-cycle
	$u_1u_2v_1v_2u_1$, we get the proper coloring $h^T$ 
	such that $h^T(u_1v_2)=h''(u_1u_2)=c_1$. Moreover,
	since these swaps only involve edges
	from $K_{n,n}$, they do not result in any ``new'' requested edges in 
	$G_1$ or $G_2$. Note that the same holds for conflict edges in $K_{2n}$.
	
	Note that the subgraph $T$,
	consisting of all edges used in the
	swaps above,
	contains two edges $u_1u_2$ and $v_1v_2$ 
	and the additional edges needed for two applications
	of Lemma \ref{swapedgeinK}; this implies that $T$ contains at
	most $2+16 \times 2 =34$ edges. Furthermore,
	except (possibly) $u_1u_2$, no edges in $T$ are prescribed; except $c_1$, 
	$T$ only contains edges with colors that are not $d$-overloaded.
	Additionally, $T$ does not contain an edge with a color in $\{t_1,\dots,t_a\}$.
	\end{proof}
	
	As for Lemma \ref{swapedgeinK}, when applying Lemma \ref{swapedgeinK1},
	we shall refer to $u_1u_2$ as the ``first edge'' and $u_1v_2$ as
	the ``second edge''.

	We use Lemma \ref{swapedgeinG} below for transforming a
	coloring $h''$ into a coloring where 
	an edge $e_y \in E(K_{n,n})$ is colored by the color $h''(e_x)$
	of an adjacent edge $e_x \in E(G_1)$ $(E(G_2))$, where
	$h''(e_x) \in \{n+1,\dots,m\}$.
	In applications of this lemma $u_1v_1$ will be referred to as the
	``first edge'', and $u_1u_2$ as the ``second edge''.
	
	%LEMMASWAPEDGEOF_G
	\begin{lemma}
	\label{swapedgeinG}
	Suppose that $h''$ is a proper $m$-edge coloring of $K_{2n}$
	obtained from $h'$ by performing some sequence of swaps
	on $h'$ and that at most $kn^2$ edges in $h''$ are disturbed
	for some constant $k>0$. Suppose further that
	for each color $c$, at most $2c(n)+P(n)$
	edges with color $c$ under $h''$ are prescribed,
	and at most $H(n)$ edges with color $c$ are disturbed.
	Moreover, let $\{t_1,\dots,t_a\}$ be a set of colors from $h''$.
	If $$\left \lfloor{\frac{n}{2}}\right \rfloor - 
	2\epsilon n - 6d n - 5 \dfrac{k+34/n^2}{d}n - 4\alpha m - 
	8c(n) - 3a - 3\beta m - 2P(n)- 6 >0$$
	and
	$$n -  (8 \dfrac{k+34/n^2}{d} n+ 2a+2+12c(n)+6\beta m + 
	8 \alpha m + 4dn + 2P(n) + 2H(n)\Big)>0$$
	then for any edge $u_1v_1$ of $G_1$ $(G_2)$
	with 
	$$h''(u_1v_1)=c_1, \; c_1 \in  \{n+1,\dots,m\}, \; 
	c_1 \notin  \{t_1,\dots,t_a\}$$ and all but at most
	$$6c(n) + 2P(n) + 2\beta m + 2\alpha m + 2a + 1 + 4 \dfrac{k+34/n^2}{d} n + 2H(n)$$ 
	choices of $u_2 \in V(G_2)$ $(V(G_1))$, there is a 
	subgraph $T$ of $K_{2n}$ and a proper $m$-edge coloring
	$h^T$, obtained from $h''$ by performing a sequence of swaps on $4$-cycles in $T$,
	that satisfies the following:
	\begin{itemize}
	
	\item except $c_1$, any color of an edge in $T$ under $h''$
	is not $d$-overloaded;
 	
	\item except $u_1v_1$, no edge in $T$ is prescribed;
	
	\item $h''$ and $h^T$ differs on at most $34$ edges 
	\emph{(}i.e. $T$ contains at most $34$ edges\emph{)};
	
	\item no edge with a color in $\{t_1,\dots,t_a\}$ under $h''$ is in $T$;
	
	\item $h^T(u_1u_2)=h''(u_1v_1)=c_1$;
	%\emph{(}in this lemma, $u_1u_2$ 
	%is called the second edge\emph{)};
	
	\item if there is a conflict of $h^T$ with $L$, 
	then this edge is also a conflict of $h''$ with $L$;
	
	\item any edge in $G_1$ or $G_2$ that is requested under $h^T$ (with respect
	to $\varphi'$) is also
	requested under $h''$.
	\end{itemize}
	\end{lemma}
	
	\begin{proof}
	Without loss of generality, assume that $u_1v_1 \in E(G_1)$. We choose 
	$u_2 \in V(G_2)$ 
	such that the following properties hold.
	\begin{itemize}
	
	\item The edge $u_2v_2 \in E(G_2)$ satisfying $h''(u_2v_2)=c_1$ 
	is not disturbed and not prescibed.
	Since there are at most $2c(n)+P(n)$ prescribed edges 
	and at most $H(n)$ disturbed edges with color $c_1$ under $h''$;
	and each prescribed or disturbed edge of $G_2$ can be 
	incident to at most two vertices of $G_2$,
	this eliminates at most $2(2c(n)+P(n)+H(n))$ choices.
	
	\item The edge $u_1u_2$ and $v_1v_2$ are both valid choices for the first 
	edge in an application of
	Lemma \ref{swapedgeinK}. As in the proof of the preceding lemma,
	this eliminates at most 
	$$2\Big(2 \dfrac{k+34/n^2}{d} n + \alpha m +c(n)+a\Big)$$ choices. 
	\item $c_1 \notin L(u_1u_2) \cup L(v_1v_2)$. 
	This excludes at most $2\beta m$ choices.
	\end{itemize}
	
	%Note that G_2 contains n vertices.
	In the coloring $h'$, there are at least $n-1$ vertices in $G_2$ that
	are incident with an edge of color $c_1$; thus we have at least
	$$n - 1 - 6c(n)- 2P(n) - 2\beta m - 2\alpha m - 
	2a - 4 \dfrac{k+34/n^2}{d} n - 2H(n)$$	
	choices for $u_2$. We note that this expression is greater 
	than zero by assumption, so we can indeed make the choice.

	Next, we want to choose a color $c_2 \in \{1,2,\dots,n\}$ 
	(which implies $c_2 \neq c_1$) such that the following properties hold.	
	\begin{itemize}
	\item The edges $e_1$ and $e_2$ colored $c_2$ under $h''$
	that are incident with $u_1$ and $v_1$, respectively,
	are both valid choices for the second edge %to be exchanged 
	in 
	an application of Lemma \ref{swapedgeinK}; this eliminates at most 
	$$2\Big(4 \dfrac{k+34/n^2}{d} n+ a+1+4c(n)+2\beta m + 2 \alpha m +2dn+P(n)\Big)$$
	choices. 
	\item $c_2 \notin L(u_1v_1) \cup L(u_2v_2)$. 
	This excludes at most $2\beta m$ choices.
	\item $c_2 \notin \varphi'(u_1) \cup \varphi'(u_2) \cup  \varphi'(v_1) \cup 
	\varphi'(v_2) \setminus \{\varphi'(u_1v_1), \varphi'(u_2v_2)\}$.
	This condition is needed to ensure that performing a series
	of swaps on $T$, does not result in a ``new'' requested edge in $G_1$ or $G_2$.
	Since there are at most $\alpha m + c(n)$ prescribed 
	edges at each vertex of $K_{2n}$ under $\varphi'$,
	this excludes at most $4(\alpha m+c(n))$ choices.

	\end{itemize}
	Thus we have at least
	$$n -  (8 \dfrac{k+34/n^2}{d} n+ 2a+2+12c(n)+6\beta m + 8 \alpha m +4dn+2P(n)\Big)$$
	choices. By assumption, this expression is greater than zero, 
	so we can indeed choose such
	color $c_2$. Now, since 
	$$\left \lfloor{\frac{n}{2}}\right \rfloor - 2\epsilon n - 
	6d n - 5 \dfrac{k+34/n^2}{d}n - 4\alpha m - 8c(n) - 3a - 3\beta m - 2P(n)- 6 >0,$$
	we can apply Lemma \ref{swapedgeinK} two consecutive 
	times to exchange the colors of $u_1u_2$ and $e_1$, and
	similarly for $v_1v_2$ and $e_2$. Finally, by swapping on the $2$-colored $4$-cycle
	$u_1u_2v_2v_1u_1$, we get the proper coloring $h^T$ such that 
	$h^T(u_1u_2)=h''(u_1v_1)=c_1$.
	
	Note that the subgraph $T$,
	consisting of all edges used in the
	swaps above,
	contains two edges $u_1v_1$ and $u_2v_2$ 
	and the additional edges needed for two applications of
	Lemma \ref{swapedgeinK}; this 
	implies that $T$ uses at most $2+16 \times 2 =34$ edges.
	Furthermore, except (possibly) $u_1v_1$, no edges in 
	$T$ are prescribed; except $c_1$,
	$T$ only contains edges with colors that are not $d$-overloaded.
	Additionally, $T$ does not contain an edge with a color in $\{t_1,\dots,t_a\}$.
	\end{proof}

	The following lemma is used for transforming the coloring $h''$
	into a coloring where an edge $e_y \in E(K_{n,n})$
	is colored by the color $h''(e_x)$ of an adjacent edge
	$e_x \in E(K_{n,n})$, where $h''(e_x) \in \{n+1,\dots,m\}$.
	When applying the lemma we shall refer to 
	$u_1u_2$ as the ``first edge'' and $u_1v_2$ as the ``second edge''.
	
	%LEMMASWAPEDGEOF_K_NEWEDGE
	\begin{lemma}
	\label{swapedgeinKnew}
	Suppose that $h''$ is a proper $m$-edge coloring of $K_{2n}$ 
	obtained from $h'$ by performing some sequence of swaps
	on $h'$ and that at most $kn^2$ edges in $h''$ are disturbed 
	for some constant $k>0$. Suppose further that
	for each color $c$, 
	at most $2c(n)+P(n)$ edges with color $c$ under $h''$ are prescribed,
	and at most $H(n)$ edges with color $c$ are disturbed.
	Let $\{t_1,\dots,t_a\}$ be a set of colors from $h''$.
	If $$\left \lfloor{\frac{n}{2}}\right \rfloor - 2\epsilon n - 
	6d n - 5 \dfrac{k+101/n^2}{d}n - 4\alpha m - 8c(n) - 3a - 3\beta m - 2P(n)- 6 >0$$
	and
	$$n -  \Big(8 \dfrac{k+101/n^2}{d} n+ 2a+2+12c(n)+6\beta m + 
	8 \alpha m +4dn+2P(n) + 2H(n)\Big)>0$$
	then for any edge $u_1u_2$ of $K_{n,n}$ with 
	$$h''(u_1u_2)=c_1, \; c_1 \in  \{n+1,\dots,m\}, \; c_1 \notin  \{t_1,\dots,t_a\}$$
	and all but at most
	$$5c(n) + 2P(n)  + \alpha m + \beta m + a + 2 + 2\dfrac{k+67/n^2}{d} n + 2H(n)$$ 
	choices of a vertex $v_2$ satisfying $u_1v_2 \in K_{n,n}$, there is a 
	subgraph $T$ of $K_{2n}$ and a proper $m$-edge coloring $h^T$, 
	obtained from $h''$ by performing a sequence of swaps on $4$-cycles in $T$, 
	that satisfies the following:
	\begin{itemize}
	
	\item except $c_1$, any color of an edge in $T$ under $h''$ is not $d$-overloaded;
 	
	\item except $u_1u_2$, no edge of $T$ is prescribed;
	
	\item $h''$ and $h^T$ differs on at most $67$ edges 
	\emph{(}i.e. $T$ contains at most $67$ edges\emph{)};
	
	\item no edge with a color in $\{t_1,\dots,t_a\}$ under $h''$ is in $T$;
	
	\item $h^T(u_1v_2)=h''(u_1u_2)=c_1$;
	%\emph{(}in this lemma, $u_1v_2$ is called the second edge\emph{)};
	
	\item if there is a conflict of $h^T$ with $L$, 
	then this edge is also a conflict of $h''$ with $L$;
	
	\item any edge in $G_1$ or $G_2$ that is requested under $h^T$ (with respect
	to $\varphi'$) is also
	requested under $h''$.
	\end{itemize}
	\end{lemma}
	
	\begin{proof}
	Without loss of generality, assume that $u_1 \in V(G_1)$;
	this implies $u_2 \in V(G_2)$. We choose $v_2 \in V(G_2)$ such that
	the following properties hold.
	\begin{itemize}
	\item The edge $v_2x \in E(G_2)$ satisfying $h''(v_2x)=c_1$ is not disturbed.
	As in the preceding lemma, this eliminates at most $2H(n)$ choices.
	
	\item The edge $v_2x$ is not prescribed and $v_2 \neq u_2$. 
	This eliminates at most $2(2c(n)+P(n))+1$ choices.
	
	\item The edge $u_1v_2$ is a valid choice for the 
	first edge %to be exchanged 
	in an application of 
	Lemma \ref{swapedgeinK}. This eliminates at most 
	$2 \dfrac{k+67/n^2}{d} n + \alpha m +c(n)+a$ choices. 
	
	\item $L(u_1v_2)$ does not contain the color $c_1$. This eliminates at most $\beta m$ choices.
	\end{itemize}
	%Before transforming from $h'$ to $h''$, 
	In the coloring $h'$,
	there are at least $n-1$ vertices in $G_2$ that
	are incident with an edge of color $c_1$; thus we have at least 
	$$n -1 - 5c(n)- 2P(n)  - \alpha m - \beta m - a -1 - 2\dfrac{k+67/n^2}{d} n - 2H(n)$$	
	choices for $v_2$. Since this expression is greater than zero by assumption,
	we can indeed make the choice.
	
	Next, we want to choose a vertex $v_1 \in V(G_1)$ satisfying the following:
	\begin{itemize}
	
	\item The edge $v_2v_1$ is a valid choice for the second edge 
	%to be transformed the color 
	in an application of Lemma \ref{swapedgeinG}. 
	%Note that this condition also implies that $c_1 \notin L(v_2v_1)$.
	This eliminates at most $$6c(n) + 2P(n) + 2\beta m + 2\alpha m +
	2a + 1+ 4 \dfrac{k+(34+67)/n^2}{d} n + 2H(n)$$ choices.
	
	\item The edge $u_2v_1$ is a valid choice for the first 
	edge %to be exchanged 
	in an application of 
	Lemma \ref{swapedgeinK} and $v_1 \neq u_1$. 
	This eliminates at most $2 \dfrac{k+67/n^2}{d} n + \alpha m + c(n) + a +1$ choices. 
	
	\item $L(u_2v_1)$ does not contain the color $c_1$. This eliminates at most $\beta m$ choices.
	\end{itemize}
	Thus we have at least $$n - 7c(n)- 2P(n)  - 3\alpha m - 3\beta m - 
	3a - 2  - 6\dfrac{k+101/n^2}{d} n - 2H(n)$$
	choices for $v_1$. Since this expression is greater than zero by assumption, 
	we can indeed make the choice.
	
	Finally, we want to choose a color $c_2 \in \{1,2,\dots,n\}$ 
	(which implies $c_2 \neq c_1$) such that the following properties hold.	
	\begin{itemize}
	
	\item The edges $e_1$ and $e_2$
	colored $c_2$ under $h''$ that
	are adjacent to $u_1$ and $u_2$, respectively,
	%and the edge $e_2$ that is adjacent to $u_2$ 
	%containing the same color $c_2$ 
	are both valid choices for the second edge %to be exchanged 
	in 
	an application of Lemma \ref{swapedgeinK}; this eliminates at most 
	$$2\Big(4 \dfrac{k+67/n^2}{d} n+ a+1+4c(n)+2\beta m + 2 \alpha m +2dn+P(n)\Big)$$
	choices. 
	\item $c_2 \notin L(u_1u_2) \cup L(v_1v_2)$. 
	This excludes at most $2\beta m$ choices.
	%$c_2 \neq c_1$ since they are in different sets of colors.
	\end{itemize}
	Thus we have at least
	$$n -  \Big(8 \dfrac{k+67/n^2}{d} n+ 2a+2+8c(n)+6\beta m + 
	4 \alpha m + 4dn + 2P(n)\Big)$$
	choices. By assumption, this expression is greater than zero, 
	so we can indeed choose such
	edges $e_1$ and $e_2$.
	
	Now, since 
	$$\left \lfloor{\frac{n}{2}}\right \rfloor - 2\epsilon n - 6d n - 
	5 \dfrac{k+101/n^2}{d}n - 4\alpha m - 8c(n) - 3a - 3\beta m - 2P(n)- 6 >0$$
	and
	$$n -  \Big(8 \dfrac{k+101/n^2}{d} n+ 2a+2+12c(n)+6\beta m + 
	8 \alpha m +4dn+2P(n) + 2H(n)\Big)>0,$$
	we can apply Lemma \ref{swapedgeinK} two consecutive times to exchange the colors 
	of $u_1v_2$ and $e_1$, and
	similarly for $u_2v_1$ and $e_2$. 
	We can thereafter apply Lemma \ref{swapedgeinG} to obtaing a coloring
	where $v_1v_2$ is colored $c_1$.
	%from $v_2x$ to $v_1v_2$.
	Now, by swapping on the $2$-colored $4$-cycle
	$u_1u_2v_1v_2u_1$, we get the proper coloring $h^T$ such 
	that $h^T(u_1v_2)=h''(u_1u_2)=c_1$.
	Since the applications of Lemma \ref{swapedgeinK}  and \ref{swapedgeinG} 
	do not result in any ``new'' requested edges in $G_1$ or $G_2$,
	the transformations in this lemma 
	do not yield any ``new'' requested edges in $G_1$ or $G_2$;
	the same holds for conflict edges in $K_{2n}$.
	
	Note that the subgraph $T$,
	consisting of all edges used in the
	swaps above
	contains an edge $u_1u_2$ and 
	all the additional edges needed for
	applying Lemma \ref{swapedgeinK} twice and Lemma \ref{swapedgeinG} once;
	this implies that $T$ contains at most $1+16 \times 2 +34=67$ edges.
	Furthermore, except $u_1u_2$, no edges in $T$ are prescribed;
	except $c_1$,
	$T$ only contains edges with colors that are not $d$-overloaded.
	Additionally,  $T$ does not contain an edge with a color in 
	$\{t_1,\dots,t_a\}$.
	\end{proof}
	
	Given a color $c_1 \in \{1,2,\dots,n\}$, the final lemma in this 
	step is used for obtaining a coloring where an edge
	in $G_1$ or $G_2$ is colored $c_1$. In applications of this
	lemma we shall refer to $uv$ as the ``first edge''.
	%transforming the color $c_1$ of an edge 
	%in $K_{n,n}$ to an edge in $G_1$ or $G_2$.
	
	%LEMMASWAPEDGEOF_COLOR_C1
	\begin{lemma}
	\label{swapedgecolor}
	Suppose that $h''$ is a proper $m$-edge coloring of $K_{2n}$ obtained 
	from $h'$ by performing some sequence of swaps
	on $h'$ and that at most $kn^2$ edges in $h''$ are disturbed for 
	some constant $k>0$. Suppose further that
	for each color $c$, at most $2c(n)+P(n)$ edges with color $c$ under $h''$ are 
	prescribed,
	and at most $H(n)$ edges with color $c$ are disturbed.
	Moreover, let $\{t_1,\dots,t_a\}$ be a set of colors from $h''$.
	If $$\left \lfloor{\frac{n}{2}}\right \rfloor - 2\epsilon n - 
	6d n - 8 \dfrac{k+104/n^2}{d}n - 4\alpha m - 15c(n) - 4a - 6\beta m - 5P(n) - 2H(n) - 6 >0$$
	then for any color $c_1 \in \{1,2,\dots,n\}$, 
	where $c_1 \notin \{t_1,\dots,t_a\}$,
	there are at least
	$$\left \lfloor{\frac{n}{2}}\right \rfloor - 7c(n) - 3P(n)  - dn - 2H(n)$$ 
	choices of an edge $uv \in E(G_1)$ $(E(G_2))$, such that there is 
	a subgraph $T$ of $K_{2n}$ and a proper $m$-edge coloring $h^T$, 
	obtained from $h''$ by performing a sequence of swaps on $4$-cycles in $T$, 
	that satisfies the
	following:
	\begin{itemize}
	
	\item except $c_1$, any color of an edge in $T$ under $h''$ is not 
	$d$-overloaded;
 	
	\item $T$ contains no prescribed edge;
	
	\item $h''$ and $h^T$ differs on at most $70$ edges 
	\emph{(}i.e. $T$ contains at most $70$ edges\emph{)};
	
	\item no edge with color in $\{t_1,\dots,t_a\}$ under $h''$ is in $T$;
	
	\item $h^T(uv)=c_1$;
	%\emph{(}in this lemma, $uv$ is called the first edge\emph{)};
	
	\item if there is a conflict of $h^T$ with $L$,
	then this edge  is also a conflict of $h''$ with $L$;
	
	\item any edge in $G_1$ or $G_2$ that is requested under $h^T$ (with respect
	to $\varphi'$) is also
	requested under $h''$.
	%in $G_1$ and $G_2$, no new requested edge is created.
	\end{itemize}
	\end{lemma}
	
	\begin{proof}
	We will prove the lemma assuming $uv \in E(G_1)$;
	the case when $uv \in E(G_2)$ is of course analogous.
	%the similar result will hold for the case $uv \in G_2$.
	Since  at most $kn^2$ edges in $h''$ are disturbed, there are at most
	$kn/d$ $d$-overloaded colors;
	by assumption, $n -1 - kn/d - a>0$, so we can choose a color
	$c_2 \in \{n+1,n+2,\dots,m\}$ such that $c_2 \notin \{t_1,\dots,t_a\}$
	is not a $d$-overloaded color. Next, we choose an edge
	$uv \in G_1$ satisfying $h''(uv)=c_2$
	such that the following properties hold.
	\begin{itemize}
	
	\item The edge $uv$ is not prescribed.
	Since there are at most $2c(n)+P(n)$ prescribed edges with
	color $c_2$ in $h''$, this eliminates at most $2c(n)+P(n)$ choices.
	
	\item The edge $uv$ is not disturbed and $c_1 \notin L(uv)$.
	 Since the color $c_2$ is not $d$-overloaded
	and for each pair of colors $c_1,c_2 \in \{1,2,\dots,m\}$, 
	there are at most $c(n)$ edges $e$
	in $K_{2n}$ with $h'(e) =c_2$ and $c_1 \in L(e)$
	%
	%color $c_2$ under $h'$ 
	%such that $c_1$ belongs 
	%to the corresponding list of forbidden colors in $L$
	and at most $dn$ edges of color $c_2$ have been used in the 
	swaps for
	transforming $h'$ to $h''$;
	this eliminates at most $c(n)+dn$ choices.
	
	\item $c_1 \notin \varphi'(u) \cup 
	\varphi'(v) \setminus \{\varphi'(uv) \}$.
	This condition is needed to ensure that after performing the 
	swaps in this lemma, 
	$uv$ is not a requested edge in $G_1$.
	Since there are at most $2c(n)+P(n)$ prescribed
	edges with color $c_1$ in $h''$,
	this excludes at most $2(2c(n)+P(n))$ choices. %count in all uv color c2
	
	\item The edges $e_1$ and $e_2$ colored $c_1$ under $h''$ that are incident
	with $u$ and $v$, respectively,
	are not disturbed. This condition implies that $e_1, e_2 \in K_{n,n}$
	and this eliminates at most $2H(n)$ choices.
	\end{itemize}
	%Before transforming from $h'$ to $h''$, 
	Under $h'$
	there are $\left \lfloor{\frac{n}{2}}\right \rfloor$ edges in $G_2$ that
	are colored $c_2$; thus we have at least
	$$\left \lfloor{\frac{n}{2}}\right \rfloor  - 7c(n)- 3P(n) - dn- 2H(n)$$	
	choices for an edge $uv$. Since this expression is 
	greater than zero by assumption, we can indeed make the choice.
	
	Next, we want to choose an edge $xy \in E(G_2)$ satisfying $h''(xy)=c_2$ 
	such that the following properties hold.
	\begin{itemize}
	
	\item $c_1 \notin L(xy) \cup \varphi'(x) \cup 
	\varphi'(y) \setminus \{\varphi'(xy)\}$,
	and the edge $xy$ is not prescibed and not disturbed.
	As before, this eliminates at most $7c(n)+3P(n) + dn$ choices.
	
	\item The edges $ux$ and $vy$ are both valid choices for the second 
	edge %to be transformed the color 
	in 
	an application of Lemma \ref{swapedgeinK1}. 
	This eliminates at most $$2\Big(4c(n) + P(n) + 2\beta m + 2\alpha m + 
	2a + 1 + 4 \dfrac{k+(34+70)/n^2}{d} n + H(n)\Big)$$ choices.
	\item $c_2 \notin L(ux) \cup L(vy)$. This eliminates at most 
	$2\beta m$ choices.
	\end{itemize}
	Thus we have at least $$\left \lfloor{\frac{n}{2}}\right \rfloor - 
	\Big(15c(n) + 5P(n) + 4\alpha m + 6\beta m + dn +  4a + 
	2 + 8\dfrac{k+104/n^2}{d} n  + 2H(n)\Big)$$	
	choices for $xy$. Since this expression is greater than zero by assumption, 
	we can indeed make the choice.
	
	Now, since 
	$$\left \lfloor{\frac{n}{2}}\right \rfloor - 2\epsilon n - 
	6d n - 8 \dfrac{k+104/n^2}{d}n - 4\alpha m - 15c(n) - 4a - 6\beta m - 5P(n) - 2H(n) - 6 >0$$
	we can apply Lemma \ref{swapedgeinK1} two consecutive times to 
	obtain a coloring where $ux$ is colored $c_1$ and $vu$ is colored $c_1$.
	%transform the color from $e_1$ to $ux$
	%and the color from $e_2$ to $vy$. 
	Thereafter, finally, by swapping on the $2$-colored $4$-cycle
	$uvyxu$, we get the proper coloring $h^T$ such that $h^T(uv)=h''(e_1)=c_1$.
	Since the swaps used when applying Lemma \ref{swapedgeinK1}
	do not result in any ``new'' requested edges in $G_1$ or $G_2$,
	the transformations in this lemma do not yield any new requested edges 
	in $G_1$ or $G_2$; similarly for conflict edges of $K_{2n}$.
		
	Note that the subgraph $T$, consisting of all edges used in the
	swaps above,
	contains two edges 
	$uv$ and $xy$ and all additional edges needed for applying
	Lemma \ref{swapedgeinK1} twice; %and Lemma \ref{swapedgeinG} once;
	this implies that $T$ contains at most $2+34 \times 2=70$ edges.
	Furthermore, none of these edges in $T$ are prescribed; except $c_1$, 
	$T$ only contains edges with colors that are not $d$-overloaded.
	Additionally, $T$ does not contain an edge with a color in $\{t_1,\dots,t_a\}$.
	\end{proof}

	%%%STEP V
	{\bf Step V:} Let $\varphi'$ be the proper $m$-precoloring of $K_{2n}$ 
	obtained in Step III and $h'$
	be the $m$-edge coloring of $K_{2n}$ obtained in Step II. 
	In this step we shall from $h'$ construct a coloring $h_q$ of $K_{2n}$
	that agrees with $\varphi$ and which avoids $L$.
	This is done iteratively by steps:
	in each step we consider a prescribed edge $e$
	of $K_{2n}$, such that $h'(e) \neq \varphi'(e)$, and
	perform a sequence of swaps on $2$-colored $4$-cycles to 
	obtain a coloring $h_e$ where
	$e$ is colored $\varphi'(e)$.
	In this process, special care is taken so that these swaps
	do not result in
	any new requested edges in $G_1$ or $G_2$;
	in particular, this implies that every requested edge
	with a color in $\{1,2,\dots,n\}$ is always in $K_{n,n}$
	for any intermediate coloring of $K_{2n}$ that is constructed in
	this iterative procedure.
	
	We shall use the following lemma.

	\begin{lemma}
	\label{correctedge}
	Suppose that $h''$ is a proper $m$-edge coloring of $K_{2n}$
	obtained from $h'$ by performing some sequence of swaps
	on $h'$ and that at most $kn^2$ edges in $h''$ are disturbed for
	some constant $k>0$.
	Suppose further that
	\begin{itemize}
	
	\item for each color $c$, at most $2c(n)+P(n)$ edges with
	color $c$ under $h''$ are prescribed;
	
	\item at most $H(n)$ edges with color $c$ are disturbed;
	
	\item all requested edges with a color
	from $\{1,2,\dots,n\}$ under $h''$ are in $K_{n,n}$.
	
	\item if $e$ is a prescribed edge of $K_{n,n}$ that satisfies
	$\varphi'(e) \neq h''(e)$, then $h''(e) \in \{1,\dots,n\}$.

	\end{itemize}
	Let $uv$ be an edge of $K_{2n}$ such that
	$$h''(uv)=c_1, \;\; \varphi'(uv)=c_2, \;\; c_1 \neq c_2.$$
	and set $$M=\left \lfloor{\frac{n}{2}}\right \rfloor - 
	\Big(2\epsilon n+24c(n) + 
	6dn + 9P(n) + 6\beta m + 4\alpha m + 10 + 8 \dfrac{k+(67+205)/n^2}{d} n + 
	6H(n)\Big)$$
	If $M>0$, then there is a subgraph $T$ of $K_{2n}$
	and a proper $m$-edge coloring $h^T$,
	obtained from $h''$ by performing a sequence of swaps on $4$-cycles in $T$, 
	that satisfies the following:
	\begin{itemize}
	
	\item $h^T(uv)=c_2$;
	
	\item $h''$ and $h^T$ differs on at most $205$ edges
	\emph{(}i.e. $T$ contains at most $205$ edges\emph{)};
	
	\item besides $uv$, $h''$ and $h^T$ disagree on at most $2$ prescribed edges;
	
	\item if $h''$ and $h^T$ disagree on a prescribed edge $ab$ 
	\emph{(}where $ab \neq uv$\emph{)}, then $ab$ is a requested edge,
	$h^T(ab)$ is not $d$-overloaded and $h''(ab) \neq \varphi'(ab)$;
	
	\item the subgraph $T$ contains at most three edges with color $c_1$ under $h''$, 
	and at most four edges with color $c_2$ under $h''$;
	
	\item except $c_1$ and $c_2$, no colors of edges in $T$ (under $h''$)
	are $d$-overloaded;
	
	\item if there is a conflict of $h^T$ with $L$, 
	then this edge is also a conflict of $h''$ with $L$;
	
	\item any edge in $G_1$ or $G_2$ that is requested under $h^T$ (with respect
	to $\varphi'$) is also
	requested under $h''$.
	%in $G_1$ and $G_2$, no new requested edge is created.

	%\item if $e$ is a prescribed edge of $K_{n,n}$ that satisfies
	%$\varphi'(e) \neq h^T(e)$, then $h^T(e) \in \{1,\dots,n\}$.
	
	%\item if $e$ is a prescribed edge of $K_{n,n}$ that satisfies
	%$\varphi(e) \in \{1,\dots,n\}$, then $h^T \in \{1,\dots,n\}$.
	%%MÅSTE SKRIVAS SENARE, ELLER? JO

	\end{itemize}
	\end{lemma}
	
	\begin{proof}
	We shall contruct a subgraph $T$ of $K_{2n}$, and by performing a
	sequence of swaps on $4$-cycles of $T$,
	we shall obtain the coloring $h^T$ from $h''$,
	where $h^T$ and $\varphi'$ agree on
	the edge $uv$. We will accomplish this
	by applying Lemmas \ref{swapedgeinK}, \ref{swapedgeinK1},
	\ref{swapedgeinG}, \ref{swapedgeinKnew}, \ref{swapedgecolor},
	and in our application of these lemmas,
	we will avoid the colors $\{c_1,c_2\}$; so $a=2$.
	
	Let $e_1$ and $e_2$ be the requested edges
	incident with $u$ and $v$, respectively, satisfying that
	$h''(e_1)= h''(e_2)=c_2$.
	%By the definition of requested edge, $h''(e_1)=h''(e_2)=c_2$.
	
	We shall consider four different cases.
	
	%Case1
	\vspace{0.5cm}
	\noindent
	\textbf{Case \pmb{$1$}}. $uv \in E(K_{n,n})$ and $c_2 \in \{1,2,\dots,n\}$:
	
	\medskip
	
	\noindent
	Since under $h''$, all requested edges with colors in
	$\{1,2,\dots,n\}$ are in $K_{n,n}$, $e_1, e_2 \in E(K_{n,n})$.
	Moreover, by assumption $c_2 \in \{1,\dots,n\}$,
	so we can proceed as in the proof of Lemma $3.7$ in
	\cite{AndrenCasselgrenMarkstrom} and use swaps on $4$-cycles,
	all edges of which are contained in $K_{n,n}$, to obtain
	a coloring $h^T$ where $h^T(uv)=c_2$. 
	Note also that this implies that every precolored edge $e$ of $K_{n,n}$
	that satisfies $h''(e) \in \{1,\dots,n\}$, also satisfies
	$h^T(e) \in \{1,\dots,n\}$.
	
	The swaps needed for obtaining the required coloring will involve
	at most $69$ edges, as described in proof of Lemma $3.7$ in
	\cite{AndrenCasselgrenMarkstrom}.
	The exact details of the transformation of the coloring $h''$ into $h^T$
	can be found in \cite{AndrenCasselgrenMarkstrom},
	so we omit them here.
	
	%Case2
	\vspace{0.5cm}
	\noindent
	\textbf{Case \pmb{$2$}}. $uv \in E(K_{n,n})$ and $c_2 \in \{n+1,n+2,\dots,m\}$:
	
	\medskip
	
	\noindent
	In this case, we will contruct a subgraph $T$ with at most $136$ edges.
	Without loss of generality, we assume that $u \in V(G_1)$, 
	this implies $v \in V(G_2)$. 
	By assumption $c_1 \in \{1,\dots,n\}$;
	we choose an edge $xy \in E(K_{n,n})$ ($x \in V(G_1)$ and $y \in V(G_2)$), 
	with $h''(xy)=c_1$
	such that the following properties hold.
	\begin{itemize}
	\item The edge $xy$ is not disturbed and not prescribed and $c_2 \notin L(xy)$. 
	Since for each pair of colors $c_1,c_2 \in \{1,2,\dots,m\}$, 
	there are at most $c(n)$ edges $e$
	in $K_{2n}$ with $h'(e) =c_1$ and $c_2 \in L(e)$, %with color $c_1$ such that 
	%$c_2$ belongs to the corresponding list of forbidden colors in $L$
	and at most $H(n)$ edges of color $c_1$ have been used 
	in the swaps for transforming $h'$ into $h''$,
	this eliminates at most $H(n)+2c(n)+P(n)+c(n)$ choices.
	\item The vertex $x$ satisfies the following.
	        \begin{itemize}
	        
					\item If $e_2 \in E(G_2)$, then we choose $x$ such that 
					$vx$ is a valid choice for the second edge 
	        %to be transformed the color 
					in an application of Lemma \ref{swapedgeinG}. 
		This eliminates at most 
		$$6c(n) + 2P(n) + 2\beta m + 2\alpha m + 5 + 4 \dfrac{k+(34+136)/n^2}{d} n +
		2H(n)$$ choices.
		
		\item If $e_2 \in E(K_{n,n})$, then since 
		$h''(e_2)=c_2 \in \{n+1,n+2,\dots,m\}$, we choose  $x$ such that $vx$ 
		is a valid choice for the second edge %to be transformed the color 
		in an application of Lemma \ref{swapedgeinKnew}.
		This eliminates at most $$5c(n) + 2P(n)  + \alpha m + \beta m +
		4 + 2\dfrac{k+(67+136)/n^2}{d} n + 2H(n)$$ choices.
		
		So in both cases, this choosing process eliminates at most
		$$6c(n) + 2P(n) + 2\beta m + 2\alpha m + 5 + 4 \dfrac{k+203/n^2}{d} n + 2H(n)$$ choices.
	        \end{itemize}
	\item The vertex $y$ is chosen with same strategy as $x$. 
	Similarly, this eliminates at most 
		$$6c(n) + 2P(n) + 2\beta m + 2\alpha m + 5 + 4 \dfrac{k+203/n^2}{d} n + 2H(n)$$ choices.
	\item $c_1 \notin L(uy) \cup L(vx)$. This excludes at most $2\beta m$ choices.
	\end{itemize}
	Thus we have at least $$n - \Big(15c(n) + 5P(n) + 6\beta m + 
	4\alpha m + 10 + 8 \dfrac{k+203/n^2}{d} n + 5H(n)\Big)$$
	choices for an edge $xy$. Since this expression is greater than zero by assumption,
	we can indeed make the choice.
	
	Now, since $M>0$, we can apply Lemma \ref{swapedgeinG} 
	or Lemma \ref{swapedgeinKnew} 
	to obtain a coloring where $uy$ is colored $h''(e_1)$.
	Similarly, we can apply Lemma \ref{swapedgeinG} 
	or Lemma \ref{swapedgeinKnew} to thereafter obtain a coloring
	where $vx$ is colored $h''(vx)$. Next, by swapping on the $2$-colored $4$-cycle
	$uvxyu$, we get the proper coloring $h^T$ such that $h^T(uv)=h''(e_1)=c_2$.
	Since the swaps 
	from the applications of 
	Lemma \ref{swapedgeinG} and Lemma \ref{swapedgeinKnew}
	do not result in any ``new'' requested edges in $G_1$ or $G_2$,
	the swaps used in this case do not no yield 
	any new requested edges in $G_1$ or $G_2$; similarly for all conflict
	edges of $K_{2n}$.
	
	Here, the subgraph $T$ contains the edges $uv$ and $xy$ and all 
	additional edges used when applying the previous lemmas above;
	in total there are at most $2+67 \times 2=136$ edges in $T$.
	
	Note further that besides $uv$ the only edges of 
	$K_{2n}$ that might be prescribed and are used
	in swaps for constructing $h^T$ are $e_1$ and $e_2$; this property
	shall be used when applying Lemma \ref{correctedge}. 
	%Thus
	%if $e$ is a prescribed edge of $K_{n,n}$ that satisfies
	%$\varphi'(e) \neq h^T(e)$, then $h^T(e) \in \{1,\dots,n\}$.
	
	%Case3
	\vspace{0.5cm}
	\noindent
	\textbf{Case \pmb{$3$}}. $uv \in E(G_1)$ (or $uv \in E(G_2)$) 
	and $c_2 \in \{1,2,\dots,n\}$:
	
	\medskip
	
	\noindent
	In this case, we shall construct a subgraph $T$ with at most $139$ edges.
	Without loss of generality, we shall assume that $uv \in E(G_1)$. Moreover,
	since  all requested edges with a  color in $\{1,2,\dots,n\}$ under $h''$ are 
	in $K_{n,n}$, $e_1, e_2 \in E(K_{n,n})$.
	%We will construct a trade on $T$ such that $T$ uses at most $139$ edges.

	If $c_1 \in \{n+1,n+2,\dots,m\}$, then we choose an 
	edge $xy \in E(G_2)$ such that $h''(xy)=c_1$
	and $xy$ is not prescribed or disturbed.
	If $c_1 \in \{1,2,\dots,n\}$, then
	we choose an edge $xy \in E(G_2)$ to be the
	first edge in an application of Lemma \ref{swapedgecolor};
	this choice implies that $xy$ is not prescribed and not disturbed. 
	So in both case we can have at least
	$$\left \lfloor{\frac{n}{2}}\right \rfloor - 7c(n) - 
	3P(n) - dn - 2H(n)$$ choices for $xy$.
	
	In addition to this, $xy$ also needs to satisfy the following:
	\begin{itemize}
	\item The edges $ux$ and $vy$ are valid choices for the second edge 
	%to be transformed the color 
	in an application of Lemma \ref{swapedgeinK1}. 
	This eliminates at most $$2\Big(4c(n) + P(n) + 
	2\beta m + 2\alpha m + 5 + 4 \dfrac{k+(34+139)/n^2}{d} n + H(n)\Big)$$ choices.	
	\item $c_2 \notin L(xy) \cup \varphi'(x) \cup \varphi'(y) 
	\setminus \{\varphi'(xy)\}$.
	This condition will imply that after performing all swaps in this case, 
	$xy$ is not a ``new'' requested edge of $G_2$.
	Since there are at most $c(n)$ edges $e$
	in $K_{2n}$ such that $h'(e) =c_1$ and $c_2 \in L(e)$, %belongs 
	%to the corresponding list of forbidden colors in $L$
	and we have already excluded the choices for $xy$ which are disturbed, 
	this condition eliminates at most $c(n) + 2(2c(n)+P(n))$ choices.
	
	\item $c_1 \notin L(ux) \cup L(vy)$. This excludes at most $2\beta m$ choices.

	\end{itemize}
	So in final, we have at least 
	$$\left \lfloor{\frac{n}{2}}\right \rfloor - \Big(20c(n) + dn + 7P(n) + 6\beta m + 4\alpha m + 10 + 8 \dfrac{k+173/n^2}{d} n + 4H(n)\Big)$$	
	choices for $xy$. Since this expression is greater than zero 
	by assumption, we can indeed make the choice.
	
	Since $M>0$, firstly if $c_1 \in \{1,2,\dots,n\}$, 
	we can apply Lemma \ref{swapedgecolor} to obtain a coloring
	where $xy$ is colored $c_1$. 
	Secondly, we can apply Lemma \ref{swapedgeinK1} twice to 
	obtain a coloring where $ux$ is colored $h''(e_1)$ and $vy$ is colored
	$h''(e_2)$. Finally, by swapping on the $2$-colored $4$-cycle
	$uvyxu$, we get the proper coloring $h^T$ such that $h^T(uv)=h''(e_1)=c_2$.
	Note that this implies that $uv$ is not a requested edge under $h^T$.
	More generally, since the swaps from the applications of 
	Lemma \ref{swapedgeinK1} and Lemma \ref{swapedgecolor}
	do not result in any new requested edges in $G_1$ or $G_2$,
	the swaps used in this case do not yield 
	any new requested edges in $G_1$ or $G_2$; similarly for conflict edges 
	in $K_{2n}$.
	
	Here, the subgraph $T$ contains the 
	edges $uv$ and $xy$ and all the additional edges needed to apply the lemmas above 
	(if we need to apply Lemma \ref{swapedgecolor}, 
	then $xy$ is included in the edges used when applying this lemma);
	in total, $T$ uses at most $1+ 70 + 34 \times 2=139$ edges.
	
	%Case4
	\vspace{0.5cm}
	
	\noindent
	\textbf{Case \pmb{$4$}}. $uv \in E(G_1)$ (or $uv \in E(G_2)$) 
	and $c_2 \in \{n+1,n+2,\dots,m\}$:
	
	\medskip
	
	\noindent
	In this case we proceed similarly to Case $3$; 
	using the same setup as in Case 3, a slight difference between two cases
	is that in Case $4$, 
	we can use Lemma \ref{swapedgeinG} or Lemma \ref{swapedgeinKnew} to
	obtain a coloring where $ux$ is colored $h''(e_1)$
	and $vy$ is colored  $h''(e_2)$. Similar calculations as above yields that 
	we have at least
	$$\left \lfloor{\frac{n}{2}}\right \rfloor - \Big(24c(n) + dn + 
	9P(n) + 6\beta m + 4\alpha m + 10 + 8 \dfrac{k+(67+205)/n^2}{d} n + 6H(n)\Big)$$	
	choices for $xy$; and this expression is greater than zero by assumption, 
	so we can indeed make the choice and perform the necessary swaps to
	get the coloring $h^T$ satisfying
	$h^T(uv)=h''(e_1)=c_2$. Here, $T$ contains at most $1+ 70 + 67 \times 2=205$ edges.
	
	\bigskip
	
	Finally, let us note that 
	in the first two cases, $T$ contains exactly two edges 
	with color $c_1$ under $h''$.
	In the last two cases, $T$ contains exactly two edges 
	with color $c_1$ under $h''$ if we
	do not have to apply Lemma \ref{swapedgecolor}; 
	otherwise $T$ contains exactly three edges with color $c_1$ under $h''$.
	Any application of Lemma \ref{swapedgeinK1}, \ref{swapedgeinG} 
	or \ref{swapedgeinKnew}
	above uses at most two edges with color $c_2$ under $h''$, so
	the subgraph $T$ contains at most four edges with color $c_2$ under $h''$.
	Except $c_1$ and $c_2$, the subgraph $T$ only contains edges with 
	colors that are not $d$-overloaded.
	\end{proof}
	
	We will take care of every prescribed edge $e$ of $K_{2n}$ 
	such that $h'(e) \neq \varphi'(e)$ by successively applying 
	Lemma \ref{correctedge}; using this lemma we can construct the
	proper $m$-edge colorings of $K_{2n}$ $h_0=h'$, 
	$h_1$, $h_2$,
	\dots, $h_q$, where $h_i$ is constructed from $h_{i-1}$ by an application of 
	Lemma \ref{correctedge} and $h_q$ is a extension
	of $\varphi'$. 
	Since the number of prescribed edge at each vertex of $K_{2n}$ 
	is at most $\alpha m+ c(n)$, the total number
	of prescribed edges in $K_{2n}$ is at most $2n(\alpha m+c(n))$, 
	thus $q \leq 2n(\alpha m+c(n))$.
	
	When we apply Lemma \ref{correctedge}, we first consider all prescribed edges $e$
	in $K_{n,n}$ that satisfies $\varphi'(e) \in \{1,\dots, n\}$ (Case 1 in the lemma).
	This is important, since otherwise we might recolor such edges by colors
	from $\{n+1,\dots, m\}$, and are thereafter unable to apply Lemma \ref{correctedge}.
	
	Thereafter we apply Lemma \ref{correctedge} to all prescribed edges $e$ of $K_{n,n}$
	that satisfies $\varphi'(e) \in \{n+1, \dots,m\}$ 
	(Case 2 in Lemma \ref{correctedge}).
	Note that after performing all the swaps as described in the preceding paragraph,
	we have not recolored any edge of $G_1$ or $G_2$.
	Thus, if one of the requested edges $e_1$ and $e_2$ 
	in Case 2 of the proof of Lemma \ref{correctedge}
	is in $K_{n,n}$, then it has been used in a previous application of
	Lemma \ref{correctedge} to a prescribed edge $e'$ of $K_{n,n}$
	that satisfies $\varphi'(e') \in \{n+1, \dots,m\}$.
	Moreover, since the only prescribed edges that are used in an application
	of Lemma \ref{correctedge} is $uv$ and (possibly) 
	the requested edges $e_1$ and $e_2$,
	it follows  that every prescribed edge $e$ in $K_{n,n}$ that satisfies
	$\varphi'(e) \in \{n+1,\dots, m\}$ is not recolored
	by a color from $\{n+1,\dots,m\} \setminus \{\varphi'(e)\}$
	in the process of applying Lemma  \ref{correctedge} to the prescribed
	edges $e$ of $K_{n,n}$ that satisfies $\varphi'(e) \in \{n+1, \dots,m\}$.
	Thus, we may assume that $h''(e) \in \{1,\dots, n\}$ for any intermediate coloring
	$h''$ and any precolored edge $e$ in $K_{n,n}$. Hence, we can perform
	all the swaps as described in Case 2 in the proof of Lemma \ref{correctedge}.
	Thereafter, we consider all prescribed edges of $G_1$ and $G_2$.

	In an application of Lemma \ref{correctedge} 
	to obtain $h_i$ from $h_{i-1}$, 
	we use swaps involving at most three prescribed edges: the edge $uv$, 
	and the two adjacent requested edges
	$e_1$ and $e_2$.
	Since there are at most $2c(n)$ prescribed edges in $K_{2n}$ 
	with any given color $c$ in $h'$,
	there are at most $\alpha m+f(n)$ prescribed edges with 
	color $c$ under $\varphi'$, and
	$h_i(e_1)$ and $h_i(e_2)$ are not $d$-overloaded colors 
	in the coloring $h_{i-1}$, %$h^T(uv)=c_2=\varphi'(uv)$;
	it follows that
	for each $i=1,\dots,q$, there are at most 
	$2c(n)+ dn+ \alpha m+f(n)$ edges with color $c$ under $h_i$ that are prescribed.
	
	Furthermore, each application of Lemma \ref{correctedge} 
	to a prescribed edge $uv$ with $h'(uv)=c$
	constructs a subgraph $T$ with at most three edges with color 
	$c$ under $h'$; thus a color $c$ is used at most 
	$3\big(2c(n)+dn+\alpha m+f(n)\big)$ times in a subgraph $T$
	where a prescribed edges has color $c$ in $h'$.
	Moreover, there are at most $\alpha m+f(n)$ prescribed edges 
	with color $c$ under $\varphi'$, and a subgraph $T$ constructed
	by an application of Lemma \ref{correctedge} uses at most four edges with color $c$.
	Except for these edges, any other edges contained in a subgraph 
	created by an application of  Lemma \ref{correctedge} are colored by 
	colors that are not $d$-overloaded. 
	Hence, at most
	$$4\big(\alpha m+f(n)\big)+3\big(2c(n)+dn+\alpha m+
	f(n)\big)+dn= 7\alpha m+7f(n)+6c(n)+4dn$$
	distinct edges with color $c$ under $h'$ are used in swaps for 
	constructing $h_q$ from $h'$.
	
	Let $H(n)=7\alpha m+7f(n)+6c(n)+4dn$, $P(n)=dn+\alpha m+f(n)$; 
	from the preceding paragraph we deduce that
	as long as $kn^2 \geq 205 \times 2n(\alpha m+c(n))=410n(\alpha m+c(n))$
	and 
	$$\left \lfloor{\frac{n}{2}}\right \rfloor - \Big(2\epsilon n + 24c(n) + 6dn + 9P(n) + 6\beta m + 4\alpha m + 10 + 8 \dfrac{k+272/n^2}{d} n + 6H(n)\Big)>0$$
	$$c'(n)=c(n)/2; \; n-1>2c(n)>4; \;
         \Big( \dfrac{4\beta}{\epsilon - 4\beta}\Big)^{\epsilon - 4\beta} \Big( \dfrac{1}{1 - 2\epsilon + 8\beta}\Big)^{1/2-\epsilon + 4\beta} <1$$
         $$\alpha, \beta < \dfrac{c(n)}{2(n-c(n))} \Big(\dfrac{n - c(n)}{n}\Big)^{\frac{n}{c(n)}}; \;
         \beta<\dfrac{c'(n)}{2(n-c'(n))} \Big(\dfrac{n - c'(n)}{n}\Big)^{\frac{n}{c'(n)}}$$
          $$m - \beta m -2 \alpha m -  2c(n) - \dfrac{2nc(n)}{f(n)} \geq 1$$ 
         for some constants $\alpha, \beta, \epsilon, k,d$ 
				and functions $c(n)$, $f(n)$ of $n$, 
         we can apply Lemma \ref{alpha} to obtain $h'$, 
				Lemma \ref{gamma} to obtain $\varphi'$ and finally Lemma  \ref{correctedge} 
         to obtain the coloring $h_q$ which is a completion 
				of $\varphi'$ that avoids $L$.
         This completes the proof of Theorem \ref{mainth}.

	\end{proof}
%$\left \lfloor{\frac{n}{2}}\right \rfloor$
		
\end{document}